\documentclass[12pt]{article}
\usepackage{amsmath}
\usepackage{amsthm}
\usepackage{amssymb}
\usepackage{amscd}
\usepackage{xspace}
\usepackage{verbatim}
\newtheorem{theorem}{Theorem}[section]
\newtheorem{corollary}{Corollary}[section]
\newtheorem{lemma}{Lemma}[section]
\newtheorem{proposition}{Proposition}[section]
\theoremstyle{definition}
\newtheorem{definition}{Definition}[section]

\theoremstyle{remark}
\newtheorem{remark}{Remark}[section]
\numberwithin{equation}{section}
\newcommand\blfootnote[1]{%
  \begingroup
  \renewcommand\thefootnote{}\footnote{#1}%
  \addtocounter{footnote}{-1}%
  \endgroup
}
\newcommand{\ov}{\overline}

\newcommand{\e}{\varepsilon}

\renewcommand{\O}{\Omega}

\renewcommand{\vec}[1]{\mathbf{#1}}
\newcommand{\field}[1]{\mathbb{#1}}

\newcommand{\R}{\field{R}}

\newcommand{\er}{\eqref}

\DeclareMathOperator{\Div}{div} \DeclareMathOperator{\dist}{dist}
\DeclareMathOperator{\supp}{supp}

\renewcommand{\O}{\Omega}

\newcommand{\f}{\varphi}
\renewcommand{\vec}[1]{\boldsymbol{#1}}
\DeclareMathOperator{\essinf}{ess\,inf}

\date{}

\begin{document}
\title{
Jump detection in Besov spaces via a new BBM
  formula. Applications to Aviles-Giga type functionals
%On spaces of bounded $q$-variation in dimension $N$
} \maketitle
\begin{center}
\textsc{Arkady Poliakovsky \footnote{E-mail:
poliakov@math.bgu.ac.il}
}\\[3mm]
Department of Mathematics, Ben Gurion University of the Negev,\\
P.O.B. 653, Be'er Sheva 84105, Israel
\\[2mm]
%\today
%\date{}
\end{center}

\begin{abstract}Motivated by the formula, due to Bourgain, Brezis and Mironescu,
\begin{equation*}
  \lim_{\e\to 0^+} \int_\Omega\int_\Omega
\frac{|u(x)-u(y)|^q}{|x-y|^q}\,\rho_\e(x-y)\,dx\,dy=K_{q,N}\|\nabla
u\|_{L^{q}}^q\,,
\end{equation*}
that characterizes the functions in $L^q$ that belong to $W^{1,q}$
(for $q>1$) and $BV$ (for $q=1$), respectively, we study what
happens when one replaces the denominator in the expression above by
$|x-y|$. It turns out that, for $q>1$ the corresponding functionals
``see'' only the jumps of the $BV$ function. We further identify the
function space relevant to the study of these functionals, the space
$BV^q$, as the Besov space $B^{1/q}_{q,\infty}$. We show, among
other things, that $BV^q(\Omega)$ contains  both the spaces
$BV(\Omega)\cap L^\infty(\Omega)$ and $W^{1/q,q}(\Omega)$. We also
present applications to the study of singular perturbation problems
of Aviles-Giga type. 
\end{abstract}

\section{Introduction}
Bourgain, Brezis and Mironescu introduced in \cite{hhh1} a new
characterization of the spaces $W^{1,q}(\Omega)$, $q>1$, and
$BV(\Omega)$ using certain double integrals involving
radial mollifiers $\{\rho_\e\}$ (see \cite{hhh1} for the precise
assumptions). In the case of a domain $\Omega\subset{\mathbb R}^N$ with
Lipschitz boundary, the so
called ``BBM formula'' states  that for any $u\in L^q(\Omega)\,(q>1)$:
\begin{equation}
\label{eq:1jjj} \lim_{\e\to 0^+} \int_\Omega\int_\Omega
\frac{|u(x)-u(y)|^q}{|x-y|^q}\,\rho_\e(x-y)\,dx\,dy=K_{q,N}\|\nabla
u\|_{L^{q}}^q\,,
\end{equation}
with the convention that $\|\nabla u\|_{L^{q}}=\infty$ if $u\notin
W^{1,q}$. For the case $q=1$ the expression in \eqref{eq:1jjj}
characterizes the $BV$-space (the latter result in its full strength
is due to  D\'{a}vila~\cite{hhh5}). For further developments in this
direction see \cite{BN,hhh2,hhh3,hhh4,hhh6}. In particular, for the
simplest choice of
\begin{equation}
\label{hgghg} \rho_\e(z)
%=\frac{1}{\e^N}\,\chi_{B_\e(0)}(z)
=\begin{cases}\frac{1}{\e^N}\,\frac{1}{\mathcal{L}^N(B_1(0))} &
z\in B_\e(0)\\0 & z\in\R^N\setminus
B_\e(0)\end{cases}\,,
%\quad\quad\forall z\in\R^N\,,
\end{equation}
we may rewrite \er{eq:1jjj} in the cases $q>1$ and $q=1$, respectively, as
\begin{align}
\label{eq:1} \lim_{\e\to 0^+} \int_\Omega\int_{
B_\e(x)\cap\Omega}\frac{1}{\e^N}
\frac{|u(x)-u(y)|^q}{|x-y|^q}\,dy\,dx&=\mathcal{L}^N(B_1(0))\,K_{q,N}\,\|\nabla
u\|_{L^{q}}^q\,,\\
\label{eq:1jjkjh} \lim_{\e\to 0^+} \int_\Omega\int_{
B_\e(x)\cap\Omega}\frac{1}{\e^N}
\frac{|u(x)-u(y)|}{|x-y|}\,dy\,dx&=\mathcal{L}^N(B_1(0))\,K_{1,N}\,\|Du\|\,.
\end{align}
We are interested in a related  formula to \eqref{eq:1}, that is
obtained when we replace $|x-y|^q$ by $|x-y|$ in the denominator
(for $q>1$). We shall
see in
our main result Theorem~\ref{ghgghgghjjkjkzzbvqred} that the resulting
formula is very
different from the one in \eqref{eq:1jjkjh}: it involves only the
``jump part'' of the gradient. We denote the space consisting of the functions
for which the resulting expression is bounded by $BV^q$. It turns out,
as we shall explain below,
that this space is closely related to the Besov Space
$B_{q,\infty}^{1/q}$.
\par A related, but different phenomenon was investigated
by Ponce and Spector in \cite{hhh3}:  for another variation on the
BBM-formula they obtained a limit where the {\em singular part} of
$Du$ appears (i.e., the sum of the jump and Cantor parts).

\par In order to state our results we shall need some  definitions.
\begin{definition}\label{gjghghghjgghGHKKzzbvq}
Given an open set $\Omega\subset\R^N$, a real number $q\geq 1$ and a
function $u\in L^q_{loc}(\Omega,\R^d)$ define:
\begin{equation}\label{GMT'3jGHKKkkhjjhgzzZZzzZZzzbvq}
\bar
A_{u,q}\big(\Omega\big):=\sup\limits_{\e\in(0,1)}\int\limits_\Omega\int\limits_{B_\e(x)\cap\Omega}\frac{1}{\e^N}\,\frac{\big|u(
y)-u(x)\big|^q}{|y-x|}dydx,
%\\=\sup\limits_{\{h\in\R^N:\,0<|h|< \dist(K,\R^N\setminus\Omega)\}}\int\limits_K\frac{1}{|h|}\Big|u(x+h)-u(x)\Big|^qdx
\end{equation}
and the infinitesimal version of this quantity:
\begin{equation}\label{GMT'3jGHKKkkhjjhgzzZZzzbvq}
\hat A_{u,q}\big(\Omega\big):=\limsup\limits_{\e\to
0^+}\int\limits_\Omega\int\limits_{B_\e(x)\cap\Omega}\frac{1}{\e^N}\,\frac{\big|u(
y)-u(x)\big|^q}{|y-x|}dydx.
%\\=\sup\limits_{\{h\in\R^N:\,0<|h|< dist(K,\R^N\setminus\Omega)\}}\Bigg(\int\limits_K\frac{1}{|h|}\Big|u(x+h)-u(x)\Big|^qdx\Bigg)
\end{equation}
\end{definition}
\begin{remark}\label{bhhfgccjjzzbvq}
It is clear that for any open $\Omega\subset\R^N$ and any $u\in
L^q_{loc}(\Omega,\R^d)$ we have
\begin{equation}\label{Ffgdfhgdhddfhyugjjkzzhhbvq}
\hat A_{u,q}\big(\Omega\big)\leq \bar A_{u,q}\big(\Omega\big).
%\quad\text{and}\quad \tilde A_{u,q}\big(\Omega\big)\leq \bar A_{u,q}\big(\Omega\big).
\end{equation}
Moreover, if $u\in L^q(\Omega,\R^d)$ then
%\begin{equation}\label{Ffgdfhgdhddfhyugkkzzbvq}
%\tilde A_{u,q,\Omega}\big(K\big)<+\infty\quad\text{if and only if}\quad A_{u,q}\big(K\big)<+\infty,
%\end{equation}
%and
%\begin{equation}\label{Ffgdfhgdhddfhyugzzbvq}
%\tilde B_{u,q,\Omega}\big(K\big)<+\infty\quad\text{if and only if}\quad B_{u,q}\big(K\big)<+\infty.
%\end{equation}
%Finally,
\begin{equation}\label{Ffgdfhgdhddfhyugkkzzhjjhhjbvq}
\bar A_{u,q}\big(\Omega\big)<\infty\quad\text{if and only if}\quad
\hat A_{u,q}\big(\Omega\big)<\infty.
\end{equation}
Clearly,
\begin{equation}\label{hjhjhghgjkghggghGHKKzzbvqkkklkljklj}
\begin{aligned}
\hat A_{u,q}\big(\R^N\big)&=\limsup\limits_{\e\to
0^+}\int\limits_{B_1(0)}\int\limits_{\R^N}\frac{1}{\e|z|}\Big|u(x+\e
z)-u(x)\Big|^qdxdz\quad\quad\text{and}\\
%\\&=\limsup\limits_{\e\to 0^+}\int\limits_{\R^N}\int\limits_{B_\e(x)}\frac{1}{\e^N}\,\frac{\big|u(y)-u(x)\big|^q}{|y-x|}dydx\,,
%\end{aligned}
%\end{equation}
%and similarly,
%\begin{equation}\label{hjhjhghgjkghggghGHKKzzbvqkkklkliih}
%\begin{aligned}
\bar
A_{u,q}\big(\R^N\big)&:=\sup\limits_{\e\in(0,1)}\int\limits_{B_1(0)}\int\limits_{\R^N}\frac{1}{\e|z|}\Big|u(x+\e
z)-u(x)\Big|^qdxdz\,.
%\\&=\sup\limits_{\e\in(0,1)}\int\limits_{\R^N}\int\limits_{B_\e(x)}\frac{1}{\e^N}\,\frac{\big|u(y)-u(x)\big|^q}{|y-x|}dydx\,.
\end{aligned}
\end{equation}
%and
%\begin{equation}\label{Ffgdfhgdhddfhyugkkiouiozziuyuyzzbvq}
%\hat A_{u,q}\big(\Omega\big)<+\infty\quad\text{implies}\quad\tilde A_{u,q}\big(\Omega\big)<+\infty.
%\end{equation}
\end{remark}
Using the quantities ${\bar A}_{u,q}$, ${\hat A}_{u,q}$ we can now define
the space $BV^q(\Omega,\R^d)$:
%(we shall see later that actually it is one the Besov spaces)
\begin{definition}\label{gjghghghjgghGHKKhjhjhjhhlkjjijhjjjkjkkzzbvq}
Given an open set $\Omega\subset\R^N$, a real number $q\geq 1$ and a
function $u\in L^q(\Omega,\R^d)$ we say that $u\in
BV^q(\Omega,\R^d)$ if
\begin{equation}\label{FfgdfhgdhddfhoGUIgfhfghhjhjhjhjhjzzbvq}
\bar A_{u,q}\big(\Omega\big)<\infty.
\end{equation}
\end{definition}
Clearly, for $u\in L^q(\Omega,\R^d)$ we have $u\in
BV^q(\Omega,\R^d)$ if and only if
\begin{equation}\label{FfgdfhgdhddfhoGUIgfhfghhjhjhjhjhjzzbvqaa}
\hat A_{u,q}\big(\Omega\big)<\infty.
\end{equation}
Moreover,
%$BV^q(\ov\Omega,\R^d)\subseteq BV^q(\Omega,\R^d)$. Moreover,
$BV^q(\Omega,\R^d)$ becomes a Banach
%normed linear
space when equipped with the norm
\begin{equation}\label{FfgdfhgdhddfhoGUIgfhfghhjhjhjhjhjyrtrtrrtzzbvq}
\|u\|_{BV^q(\Omega,\R^d)}:=\Big(\bar
A_{u,q}\big(\Omega\big)\Big)^{\frac{1}{q}}+\|u\|_{L^q(\Omega,\R^d)}.
\end{equation}
Next, given a function $u\in L^q_{loc}(\Omega,\R^d)$ we say that
$u\in BV^q_{loc}(\Omega,\R^d)$ if for every open
$\Omega'\subset\subset\Omega$ we have  $u\in BV^q(\Omega',\R^d)$.

%The name $BV^q$ for the space will be clarified in Theorem
%\ref{ghgghgghjjkjkzzbvq} below, which seems to be the most important
%result of the present paper.
\begin{remark}\label{bhhfgcczzbvq}
By  the ``BBM formula" we have
$BV^1_{loc}(\Omega,\R^d)=BV_{loc}(\Omega,\R^d)$ and in the case of a
domain $\Omega$ with Lipschitz boundary, also
$BV^1(\Omega,\R^d)=BV(\Omega,\R^d)$.
\end{remark}

In our main result, Theorem~\ref{ghgghgghjjkjkzzbvqred}, we prove an
explicit formula for $\hat A_{u,q}\big(\Omega\big)$ when $u\in BV(\Omega,\R^d)\cap
L^\infty(\Omega,\R^d)$. This formula justifies the name we have chosen
 for
the space $BV^q$.
\begin{theorem}\label{ghgghgghjjkjkzzbvqred}
Let $\Omega\subset\R^N$ be an open set with bounded Lipschitz
boundary and let $u\in BV(\Omega,\R^d)\cap L^\infty(\Omega,\R^d)$.
Then, for every $q>1$ we have $u\in {BV}^q(\Omega,\R^d)$ and
\begin{equation}\label{fgyufghfghjgghgjkhkkGHGHKKjjjjkjkkjkmmlmjijiluuizziihhhjbvqKKred}
\hat
A_{u,q}\big(\Omega\big)=C_N\int_{J_u}\Big|u^+(x)-u^-(x)\Big|^qd\mathcal{H}^{N-1}(x).
%=A_{u,q}\big(\Omega\big).
\end{equation}
with the dimensional constant $C_N>0$ defined by
\begin{equation}\label{fgyufghfghjgghgjkhkkGHGHKKggkhhjoozzbvqkk}
C_N:=\frac{1}{N}\int_{S^{N-1}}|z_1|d\mathcal{H}^{N-1}(z)\,,
\end{equation}
where we denote $z:=(z_1,\ldots, z_N)\in\R^N$.
\end{theorem}
%
%
%
\begin{comment}
\begin{theorem}\label{ghgghgghjjkjkzzbvq}
Let $\Omega\subset\R^N$ be an open set with bounded Lipschitz
boundary and let $u\in BV(\Omega,\R^d)\cap L^\infty(\Omega,\R^d)$.
Then, for every $q>1$ we have $u\in {BV}^q(\Omega,\R^d)$ and
\begin{equation}\label{fgyufghfghjgghgjkhkkGHGHKKjjjjkjkkjkmmlmjijiluuizziihhhjbvq}
\hat
A_{u,q}\big(\Omega\big)=C_N\int_{J_u}\Big|u^+(x)-u^-(x)\Big|^qd\mathcal{H}^{N-1}(x),
%=A_{u,q}\big(\Omega\big),
\end{equation}
with the dimensional constant $C_N>0$ defined by
\begin{equation}\label{fgyufghfghjgghgjkhkkGHGHKKggkhhjoozzbvq}
C_N:=\frac{1}{N}\int_{S^{N-1}}|z_1|d\mathcal{H}^{N-1}(z)\,,
\end{equation}
where we denote $z:=(z_1,\ldots, z_N)\in\R^N$.
\end{theorem}
\end{comment}
%
%
%
%As we shall see in Section~\ref{hjhjghjyh},
%Theorem~\ref{ghgghgghjjkjkzzbvq} can be deduced from a more general
%result, Theorem~\ref{ghgghgghjjkjkzzbvqKK}.
\begin{remark}\label{ghuhghffh}
Note the {\bf big difference} between the case $q>1$ and $q=1$.
Indeed, by \eqref{eq:1jjkjh} for ${BV^1}={BV}$ the analog of
\eqref{fgyufghfghjgghgjkhkkGHGHKKjjjjkjkkjkmmlmjijiluuizziihhhjbvqKKred} is
$$\hat
A_{u,1}\big(\Omega\big)=\mathcal{L}^N(B_1(0))\,K_{1,N}\,\|Du\|(\Omega),$$
that is, for $q=1$ we see the {\bf full} $BV$-seminorm, not just the
``jump part''!
\end{remark}
Our next result deals with functions in $W^{\frac{1}{q},q}$:
\begin{theorem}\label{huyhuyuughhjhhjjhhjjjhhjhjjhhhjzzbvq}
Given an open set $\Omega\subset\R^N$, $q\geq 1$ and a
function $u\in W^{\frac{1}{q},q}(\Omega,\R^d)$ we have $u\in
BV^{q}(\Omega,\R^d)$, and if in addition $q>1$, then
$
%B_{u,q}\big(\Omega\big)=A_{u,q}\big(\Omega\big)=
\hat
A_{u,q}\big(\Omega\big)=0$. Moreover, the embedding $
W^{\frac{1}{q},q}(\Omega,\R^d)\subset BV^{q}(\Omega,\R^d)$ is
continuous.
\end{theorem}

Next we recall the definition of the Besov Spaces $B_{q,\infty}^s$ with
$s\in(0,1)$:
\begin{definition}\label{gjghghghjgghGHKKhjhjhjhhzzbvqkkl,.,.}
Given $q\geq 1$ and $s\in(0,1)$, we say that $u\in
L^q(\mathbb{R}^N,\R^d)$ belongs to the Besov space
$B_{q,\infty}^s(\mathbb{R}^N,\R^d)$ if
\begin{equation}\label{gjhgjhfghffh}
\sup\limits_{\rho\in(0,\infty)}\Bigg(\sup_{|h|\leq\rho}\int_{\mathbb{R}^N}
\frac{|u(x+h)-u(x)\big|^q}{\rho^{sq}}dx\Bigg)<\infty.
\end{equation}
Moreover, for every open $\Omega\subset\R^N$ we say that $u\in
L^q_{loc}(\Omega,\R^d)$ belongs to Besov space
$\big(B_{q,\infty}^s\big)_{loc}(\Omega,\R^d)$ if for every compact
$K\subset\subset\Omega$ there exists $u_K\in
B_{q,\infty}^s(\mathbb{R}^N,\R^d)$ such that $u_K(x)= u(x)$ for
every $x\in K$.
\end{definition}
The next result clarifies the relation between the space $BV^q$ and Besov spaces:
\begin{proposition}\label{ghgghgghzzbvqhjhjh}
For $q>1$ we have:
\begin{equation}\label{hjhjjggjklkl}
BV^q(\R^N,\R^d)=B_{q,\infty}^{1/q
%\frac{1}{q}
}(\R^N,\R^d).
\end{equation}
Moreover for every open $\Omega\subset\R^N$ and $q>1$ we have:
\begin{equation}\label{hjhjjggjjkjk}
BV^q_{loc}(\Omega,\R^d)=\big(B_{q,\infty}^{ 1/q
%\frac{1}{q}
}\big)_{loc}(\Omega,\R^d).
\end{equation}
%%In particular, $BV^q(\R^N,\R^d)=B_{q,\infty}^{1/q
%\frac{1}{q}
%%}(\R^N,\R^d)$ and $BV^q_{loc}(\Omega,\R^d)=\big(B_{q,\infty}^{ 1/q
%\frac{1}{q}
%%}\big)_{loc}(\Omega,\R^d)$ $\;\forall q>1$.
%$\hat{BV}^q(\Omega,\R^d)=BV^q(\Omega,\R^d)$.
\end{proposition}
We should mention that \eqref{hjhjjggjklkl} of Proposition
\ref{ghgghgghzzbvqhjhjh} can be deduced from a more general result,
obtained independent by Brasseur in \cite{Br}, that characterizes the
Besov spaces $B_{p,\infty}^s(\R^N)$ via a BBM-type formula, for all
values of $s\in(0,1)$ and $p\in[1,\infty)$.

\begin{remark}
Similar results hold also for more general mollifiers than in \eqref{hgghg}, e.g., of the form
$\rho_\e(x)=\frac{\rho(|x|/\e)}{\e^N}$, where $\rho(t)$ is a nonnegative
function on $[0,\infty)$ with compact support, such that
$\essinf_{(0,\delta)} \rho\geq \alpha$ for some $\alpha,\delta>0$ and  $\int_{\R^N}\rho(|x|)\,dx=1$. We
did not investigate more general families of radial mollifiers
$\{\rho_\e(x)\}$ as in \cite{hhh1}.
\end{remark}

In \cite{BBM3} Bourgain, Brezis and Mironescu introduced a new
space, that they called $B$, that contains the spaces $BV$,
$W^{\frac{1}{q},q}$ for $q\geq 1$ and $BMO$. Moreover, they
introduced a proper subspace $B_0\subsetneq B$, such that $B_0$
contains $W^{\frac{1}{q},q}$ for $q\geq 1$ as well as $VMO$. For every
$u\in B$ they defined the seminorm $|u|_{B}$
and its infinitesimal version $[u](\Omega)$. The precise definitions
are given bellow in Definition~\ref{gvyfhgfhfffbvq}. Our next result
deals with  the relations between the  $BV^q$ spaces and the spaces $B$
and $B_0$:
\begin{theorem}\label{bhugggyffyfbvq}
Let $\Omega\subset\R^N$ be an open set and $q\geq 1$. Then for every
$u\in BV^q(\Omega,\R^d)$ we have
\begin{equation}\label{jghjghjgghghghghjhkhhhihyuhyiibvquiiu}
|u|_{B(\Omega,\R^d)}\leq N^{\frac{N+1}{2q}}\big(\bar
A_{u,q}(\Omega)\big)^{\frac{1}{q}}.
\end{equation}
and
\begin{equation}\label{jghjghjgghghghghjhkhhhihyuhyiibvq}
[u](\Omega)\leq N^{\frac{N+1}{2q}}\big(\hat
A_{u,q}(\Omega)\big)^{\frac{1}{q}}.
\end{equation}
Moreover, if in addition $\mathcal{L}^N(\Omega)<\infty$ then
$BV^q(\Omega,\R^d)\subset B(\Omega,\R^d)$ with continuous embedding.
In particular, if $\hat A_{u,q}(\Omega)=0$ then $u\in
B_0(\Omega,\R^d)$.
\end{theorem}

We now turn to the role  of  $BV^q$-spaces in the study of  singular
perturbation problems.  In various applications one is led to study
the $\Gamma$-limit, as $\varepsilon\to 0^+$, of the Aviles-Giga
%a family of
functional $I^{(2)}_\varepsilon$, defined for scalar functions
$\psi$ by
\begin{equation}
\label{b5..}
I^{(2)}_\varepsilon(\psi):=\int_\Omega\left\{\varepsilon|\nabla^2
\psi|^2+\frac{1}{\varepsilon}\Big(1-|\nabla
\psi|^2\Big)^2\right\}dx\quad
%\text{for}\;\,u:\Omega\to\mathbb{R}\;\,\text{such that}\;\,u=0,\,\frac{\partial u}{\partial \vec n}=-1\,\;\text{on}\,\;\partial\Omega
\quad\text{(see \cite{adm,ag1,ag2})}.
\end{equation}
Here $\Omega\subset\R^N$ is a bounded domain.
A generalization of \eqref{b5..}  to any  $p>1$ is:
\begin{equation}\label{bjhgjh}
I^{(p)}_\varepsilon(\psi):=\int_\Omega\bigg(\varepsilon^{p-1}|\nabla^2
\psi|^p+\frac{1}{\varepsilon}\Big(1-|\nabla
\psi|^2\Big)^{\frac{p}{p-1}}\bigg)=\int_\Omega\frac{1}{\varepsilon}\bigg(\big|\varepsilon\nabla^2
\psi\big|^p+\Big(1-|\nabla \psi|^2\Big)^{\frac{p}{p-1}}\bigg).
\end{equation}

It is clear that the functional $(\Gamma-\limsup_{\varepsilon\to
0^+}I^{(p)}_\varepsilon)(\psi)$, calculated in the strong $W^{1,q}$
topology, can be finite only if
%\begin{equation}\label{hghiohoijojjkhhhjhjjkjgg}
$|\nabla\psi|^2=1\;\,\text{for a.e.}\;\,x\in\Omega$,
%\end{equation}
i.e., if we define:
\begin{equation}
\begin{aligned}
\label{cuyfyugugghvjjhh}
\mathcal{A}_{0}:=\mathcal{A}_{0}(\Omega,q)&:=\Big\{\psi\in
W^{1,q}(\Omega)\,:\; \big|\nabla\psi(x)\big|^2=1\;\,\text{for
a.e.}\,\;x\in\Omega\Big\},
%\end{equation*}
%\begin{equation}\label{cuyfyugugghvjjhhggihug}
\\
\mathcal{A}:=\mathcal{A}(\Omega,q)&:=\Big\{\psi\in
W^{1,q}(\Omega)\,:\; (\Gamma-\limsup_{\varepsilon\to 0^+}
I^{(p)}_\varepsilon\big|_\Omega)(\psi)<\infty\Big\},
\end{aligned}
\end{equation}
then clearly $\mathcal{A}\subset\mathcal{A}_0$. Note that the set
$\mathcal{A}$ consists of functions with discontinuous gradients.
The natural space of discontinuous functions is $BV$ space. It turns
out that we have
$\mathcal{A}_{BV}\subset\mathcal{A}\subset\mathcal{A}_0$, where
$\mathcal{A}_{BV}:=\mathcal{A}_0\cap \{\psi:\,\nabla\psi\in BV\}$.
However, Ambrosio, De Lellis and Mantegazza showed in \cite{adm}
that $\mathcal{A}_{BV}\subsetneq\mathcal{A}$ in the special  case
of the energy \eqref{b5..} when $N=2$. On the other hand, as
 shown by Camillo de Lellis and Felix Otto in \cite{CDFO}, for the
energy \eqref{b5..} the set $\mathcal{A}$ is contained in a certain
space of functions that still inherits some good geometric measure
theoretical properties of $BV$ space.

A lower bound for \eqref{b5..} when $N=2$ was found
 by Aviles and Giga in \cite{ag2},
by Jin and Kohn in \cite{jin},
  and by
Ambrosio, De Lellis and Mantegazza in \cite{adm}. A matched upper
bound, in the case $\nabla \psi\in BV$, was found independently by
Conti and De Lellis~\cite{CdL} and Poliakovsky~\cite{pol}. These
results imply that for the particular case $\nabla \psi\in BV$ and
$N=2$, the $\Gamma$-limit functional of \eqref{b5..}, calculated in
the strong $W^{1,q}$-topology, is
\begin{equation}\label{hfighfighfihhjjhhjlkjj}
\tilde I_{0}(\psi):=
\begin{cases}
\frac{1}{3}\int_{J_{\nabla \psi}}\big|\nabla \psi^+(x)-\nabla
\psi^-(x)\big|^3\,d\mathcal{H}^{N-1}(x)\quad\quad\text{if }\;|\nabla
\psi|=1\;\;\text{a.e. in }\;\Omega
%\nabla u\in\mathcal{A}_{BV},
\\+\infty\quad\quad\text{otherwise}.
\end{cases}
\end{equation}
This results can also be generalized to show that, up to a
multiplicative constant, the energy \er{hfighfighfihhjjhhjlkjj} is
also the $\Gamma$-limit of functional \eqref{bjhgjh}.
Indeed, the
lower bound for \eqref{bjhgjh} can be obtained analogously to that
for \eqref{b5..}, using H\"{o}lder inequality instead of
Cauchy-Schwarz and the matched upper bound can be obtained as a
special  case of a more general result, obtained in \cite{polgen}.
However,
%Ambrosio, De Lellis and Mantegazza showed in \cite{adm} that
as we already mentioned,
%the limiting space, i.e., the set of those functions $\nabla \psi$ where the $\Gamma$-limit is finite, is strictly larger than $BV$-space
we have $\mathcal{A}_{BV}\neq\mathcal{A}$ for problem \eqref{b5..}
and thus the question of the value of the $\Gamma$-limit  in the
case $\nabla \psi\notin BV$ is still open.
\par We also recall that De Lellis showed in \cite{dl} that for $N=3$ and
$\nabla \psi\in BV$, the functional \er{hfighfighfihhjjhhjlkjj} is
not lower semicontinuous in the $L^1$-topology and thus cannot by
the $\Gamma$-limit of \eqref{b5..}.
%As we shall see below,
%the space $BV^3$ (the case $q=3$ of the space $BV^q$), appears
%naturally in the study of the Aviles-Giga functional \er{b5..} or more generally of the functional \eqref{bjhgjh}.

In the particular case of the functional \er{b5..} with $N=2$ we
propose here  a candidate for the set $\mathcal{A}$, namely the set
$\big\{\psi:\Omega\to\R:\,\,\nabla \psi\in BV^3,\;|\nabla
\psi|=1\big\}$ (where $BV^3$ is the case $q=3$ of the space $BV^q$).
Indeed, by Theorem \ref{ghgghgghjjkjkzzbvqred} and
\eqref{hfighfighfihhjjhhjlkjj}, when $N=2$, $|\nabla\psi|=1$ and
$\nabla \psi\in BV$, the $\Gamma$-limit of  the functional \er{b5..}
equals $\big(\frac{1}{3C_3}\big) \hat A_{\nabla
\psi,3}\big(\Omega\big)
%=\big(\frac{1}{3C_3}\big) A_{\nabla\psi,3}\big(\Omega\big)
$.
% \begin{equation}\label{fgyufhjjh}
% \hat A_{\nabla \psi,3}\big(\Omega\big)=A_{\nabla
% \psi,3}\big(\Omega\big)=C_3\int_{J_{\nabla \psi}}\big|\nabla
% \psi^+-\nabla \psi^-\big|^3d\mathcal{H}^{N-1}
% \end{equation}
Therefore, it is natural to conjecture that
$\big(\frac{1}{3C_3}\big) \hat A_{\nabla \psi,3}\big(\Omega\big)$ is
the $\Gamma$-limit also in the case $\nabla \psi\notin BV$, and more
specifically that $\mathcal{A}=\big\{\psi:\Omega\to\R:\,\,\nabla
\psi\in BV^3,\;|\nabla \psi|=1\big\}$. We have an analogous
conjecture  for the functional \er{bjhgjh}, with a different
constant multiplying $\hat A_{\nabla \psi,3}\big(\Omega\big)$. An
additional suport for this conjecture is provided by the fact that
the example constructed by Ambrosio, De Lellis and Mantegazza in
\cite{adm}, of a function $\psi\in
\mathcal{A}\setminus\mathcal{A}_{BV}$,
%in the particular case of the energy defined by \eqref{b5..} with $N=2$,
turns out to satisfy $\psi\in BV^3$ (as it can be easily verified).

Our next result provides a (non-sharp) upper bound for a more
general energy than the one in \er{bjhgjh}:
\begin{theorem}\label{hgughgfhfzzbvq}
Given an open set $\Omega\subset\R^N$, let
$\Omega_0\subset\subset\Omega$ be a compactly embedded open subset
and $\psi\in W^{1,\infty}_{loc}(\Omega,\R)$ be such that
$|\nabla\psi(x)|=1$ for a.e. $x\in\Omega$.
% and $\nabla\psi(x)\in BV^3_{loc}(\Omega,\R^N)$.
Let $\eta\in C^\infty_c(\R^N,\R)$ be a nonnegative function such
that  $\supp\eta\subset\ov B_1(0)$ and $\int_{\R^N}\eta(z)dz=1$. For
every $x\in\Omega$ and every $0<\e<\dist(x,\partial\Omega)$ define
\begin{equation}\label{jmvnvnbccbvhjhjhhjjkhgjgGHKKjhhjzzbvq}
\psi_\e(x):=\frac{1}{\e^N}\int_{\R^N}\eta\Big(\frac{y-x}{\e}\Big)\psi(y)dy=\int_{\R^N}\eta(z)\psi(x+\e
z)dz.
\end{equation}
Assume in addition that  $\nabla\psi(x)\in BV^q_{loc}(\Omega,\R^N)\cap
BV^p_{loc}(\Omega,\R^N)$ for some $q>1$ and $p\geq 2$. Then we have:
\begin{multline}\label{fgyufghfghjgghgjkhkkhhkggkjgmhkjjuyhghgGHKKvbvbvvhughghojjlkjjjlouokjjzzbvqaajjkljkjjkjkjkjkh}
\limsup_{\e\to
0^+}\Bigg(\int_{\Omega_0}\e^{q-1}\big|\nabla^2\psi_\e(x)\big|^qdx+\frac{1}{\e}\Big(1-\big|\nabla\psi_\e(x)\big|^2\Big)^{\frac{p}{2}}dx\Bigg)\leq\\
\bigg(\int_{\R^N}|z|^{\frac{1}{q-1}}\big|\nabla\eta(z)\big|^{\frac{q}{q-1}}dz\bigg)^{q-1}A_{\nabla\psi,q}(\ov\Omega_0)+\bigg(\int_{\R^N}|z|^{\frac{2}{p-2}}
\big|\eta(z)\big|^{\frac{p}{p-2}}dz\bigg)^{\frac{p-2}{2}}A_{\nabla\psi,p}(\ov\Omega_0),
%\Bigg(\int_{B_1(0)}\int_{\Omega_0}\frac{1}{\e|z|}\Big|\nabla\psi(x+\e z)-\nabla\psi(x)\Big|^qdxdz\Bigg).
\end{multline}
where
\begin{equation}\label{hjhjhghgjkghggghGHKKzzbvqkkljljkj}
A_{\nabla\psi,\rho}\big(\ov\Omega_0\big):=\limsup\limits_{\e\to
0^+}\int\limits_{\ov\Omega_0}\int\limits_{B_\e(x)}\frac{1}{\e^N}\,\frac{\big|\nabla\psi(
y)-\nabla\psi(x)\big|^\rho}{|y-x|}dydx\,.
\end{equation}
In particular, if $\nabla\psi(x)\in BV^3_{loc}(\Omega,\R^N)$ then:
\begin{multline}\label{fgyufghfghjgghgjkhkkhhkggkjgmhkjjuyhghgGHKKvbvbvvhughghojjlkjjjlouokjjgghgjgokjjlkzzbvq}
\frac{3}{\sqrt[3]{4}}\,\limsup_{\e\to
0^+}\Bigg(\int_{\Omega_0}\big|\nabla^2\psi_\e(x)\big|\Big|1-\big|\nabla\psi_\e(x)\big|^2\Big|dx\Bigg)\leq\\
\limsup_{\e\to 0^+}\Bigg(\int_{\Omega_0}\e^2\big|\nabla^2\psi_\e(x)\big|^3dx+\int_{\Omega_0}\frac{1}{\e}\Big(1-\big|\nabla\psi_\e(x)\big|^2\Big)^{\frac{3}{2}}dx\Bigg)\leq\\
D_\eta A_{\nabla\psi,3}(\overline\Omega_0)=D_\eta\limsup_{\e\to
0^+}\Bigg(\int_{B_1(0)}
\int_{\Omega_0}\frac{1}{\e|z|}\Big|\nabla\psi(x+\e
z)-\nabla\psi(x)\Big|^3dxdz\Bigg),
\end{multline}
where the constant $D_\eta$ is given by
\begin{equation}\label{jmvnvnbccbvhjhjhhjjkhgjgGHKKjhhjuyugkhhklzzbvq}
D_\eta:=\bigg(\int_{\R^N}|z|^{\frac{1}{2}}\big|\nabla\eta(z)\big|^{\frac{3}{2}}dz\bigg)^{2}+\bigg(\int_{\R^N}|z|^{2}\big|\eta(z)\big|^{3}dz\bigg)^{\frac{1}{2}}.
\end{equation}
\end{theorem}
As a direct consequence of the last Theorem we extend the previously
known result about the boundedness of the
$\Gamma-\limsup$ for the energy in \er{bjhgjh} when $p=3$ from the
case $\nabla\psi\in BV$ (see \cite{polgen}) to the case
$\nabla\psi\in BV^3$:
\begin{corollary}\label{ghgghg}
Given an open set $\Omega\subset\R^N$, let $\psi\in
W^{1,\infty}_{loc}(\Omega,\R)$ be such that $|\nabla\psi(x)|=1$ for
a.e. $x\in\Omega$ and $\nabla\psi(x)\in BV^3_{loc}(\Omega,\R^N)$.
Then, for every compactly embedded open subset
$\Omega'\subset\subset\Omega$ and every $q\geq 1$ we have, $\psi\in
\mathcal{A}(\Omega',q)$, with $\mathcal{A}(\Omega',q)$ given by
\begin{equation}
\mathcal{A}(\Omega',q):=\Big\{\psi\in W^{1,q}(\Omega')\,:\,
(\Gamma-\limsup_{\varepsilon\to 0^+}
I^{(3)}_\varepsilon\big|_{\Omega'})(\psi)<+\infty\;
(\text{calculated in the
%strong
$W^{1,q}$ topology})\Big\},
\end{equation}
where $I^{(3)}_\e$ is given by \er{bjhgjh} with $p=3$. Moreover, we
have
\begin{equation}\label{bhghgghg}
(\Gamma-\limsup_{\varepsilon\to 0^+}
I^{(3)}_\varepsilon\big|_{\Omega'})(\psi)\leq C
A_{\nabla\psi,3}(\overline\Omega'),
\end{equation}
for some constant $C>0$.
\end{corollary}
\begin{remark}
We do not know whether one can get  a global and  sharp ``improved''
version of  Corollary \ref{ghgghg} with $\Omega'=\Omega$ and with
the constant $C:=\frac{1}{2\sqrt[3]{4}\,C_N}$ in \er{bhghgghg}. This
is the sharp constant for the energy \er{bjhgjh} with $p=3$ and
$N=2$ in the particular case where $\nabla\psi\in BV$.
\end{remark}

%obtained by myself in \cite{polgen}

The paper is organized as follows. Section~\ref{ghbyhgvyhvg} is
devoted to definitions and properties of the spaces $BV^q$. In
subsection~\ref{hggyffgfdfdfd} we present some additional
definitions and generalized   versions of some of the results stated
above. In subsection~\ref{hjhjghjyh} we give the proofs of our main
results about the spaces $BV^q$. In Section~\ref{AVGG} we give the
proof of Theorem \ref{hgughgfhfzzbvq}, which is an application of
the spaces $BV^q$ to the study of energies of Avies-Giga type. The
proofs of Proposition~\ref{gygfhffghlkoii} and Lemma \ref{hjgjg} are
given in the Appendix \ref{AppB}. For the convenience of the reader,
in Appendix \ref{AppA} we states some known results on $BV$
functions, that we need for the proof.

%The necessary definition and basic known results about $BV$-functions, that we use throughout the paper, are presented in Appendix B.
\subsubsection*{Acknowledgments} The research was supported by  the
Israel Science Foundation (Grant No. 999/13). I thank Itai Shafrir
for some interesting  discussions and Petru Mironescu for very
helpful suggestions that helped me improve an earlier version of the
manuscript.
%\subsubsection*{Acknowledgments} The research was supported by  the
%Israel Science Foundation (Grant No. 999/13). I thank Itai Shafrir for
%many interesting  discussions.

\section{Properties of the space $BV^q$}\label{ghbyhgvyhvg}
\subsection{
%Supplementations to Proposition \ref{ghgghgghzzbvqhjhjh}
Some additional definitions and results}\label{hggyffgfdfdfd}
First we introduce local versions of the quantity
%$\bar A_{u,q}$ and
${\hat A}_{u,q}$
%, define the related quantity $B_{u,q}$
that are related to the space $BV^q_{loc}$:
\begin{definition}\label{gjghghghjgghGHKKzzbvqKK}
Given a compact set $ \ov U\subset\subset\O$ let
\begin{equation}\label{hjhjhghgjkghggghGHKKzzbvq}
\begin{aligned}
A_{u,q}\big(\ov U\big)&:=\limsup\limits_{\e\to
0^+}\int\limits_{B_1(0)}\int\limits_{\ov
U}\frac{1}{\e|z|}\Big|u(x+\e
z)-u(x)\Big|^qdxdz\\
&=\limsup\limits_{\e\to 0^+}\int\limits_{\ov
U}\int\limits_{B_\e(x)}\frac{1}{\e^N}\,\frac{\big|u(
y)-u(x)\big|^q}{|y-x|}dydx\,.
\end{aligned}
\end{equation}
For an open set $\Omega\subset\R^N$ let
\begin{equation}
A_{u,q}\big(\Omega\big):=\sup\limits_{K\subset\subset\Omega}A_{u,q}\big(K\big).\label{hjhjhghgjkghggghGHKKjjzzbvq}
\end{equation}
\end{definition}
\begin{remark}\label{bhhfgccjjzzbvqKK}
It is clear that for any open $\Omega\subset\R^N$, any $u\in
L^q_{loc}(\Omega,\R^d)$ and for any compactly embedded open set
$\Omega_0\subset\subset\Omega$ we have
\begin{equation}\label{Ffgdfhgdhddfhyugkkiouiozzhjhjbvq}
\hat A_{u,q}\big(\Omega_0\big)\leq A_{u,q}\big(\ov\Omega_0\big)\leq
A_{u,q}\big(\Omega\big)\leq\hat A_{u,q}\big(\Omega\big).
\end{equation}
%\begin{equation}\label{Ffgdfhgdhddfhyugkkiouiozzbvq}
%A_{u,q}\big(K\big)\leq\tilde A_{u,q,\Omega}\big(K\big)\leq \tilde A_{u,q}\big(\Omega\big),
%%\leq \hat A_{u,q}\big(\Omega\big),
%\end{equation}
%\begin{equation}\label{Ffgdfhgdhddfhyugjjkzzbvq}
%B_{u,q}\big(K\big)\leq\tilde  B_{u,q,\Omega}\big(K\big)\leq \tilde B_{u,q}\big(\Omega\big),
%\end{equation}
\end{remark}
%Furthermore, we say that $u\in \hat{BV}^q(\Omega,\R^d)$ if $u\in L^q(\Omega,\R^d)$ and
%\begin{equation}\label{FfgdfhgdhddfhoGUIgfhfghhjhjzzbvq}
%A_{u,q}\big(\Omega\big)=\sup\limits_{K\subset\subset\Omega}A_{u,q}\big(K\big)<+\infty.
%\end{equation}
\begin{remark}\label{gjghghghjgghGHKKhjhjhjhhzzbvq}
%Clearly $u\in {BV}^q_{loc}(\Omega,\R^d)$  if and only if for every
%compactly embedded open subset $\Omega'\subset\subset\Omega$ we have
%$u\in {BV}^q(\Omega',\R^d)$.
Clearly, given an open set $\Omega\subset\R^N$, $q\geq 1$ and a
function $u\in L^q_{loc}(\Omega,\R^d)$ we have
$u\in{BV}^q_{loc}(\Omega,\R^d)$
%(respectively, $u\in\widehat{BV}^q_{loc}(\Omega,\R^d)$)
if and only if for every compact subset $K\subset\subset\O$ we have
$A_{u,q}\big(K\big)<\infty$.
%(respectively, $B_{u,q}\big(K\big)<\infty$).
%Furthermore, we say that $u(x)\in BV^q(\Omega,\R^d)$ if $u(x)\in L^q(\Omega,\R^d)$ and
%\begin{equation}\label{FfgdfhgdhddfhoGUIgfhfgzzbvq}
%B_{u,q}\big(\Omega\big)=\sup\limits_{K\subset\subset\Omega}B_{u,q}\big(K\big)<+\infty.
%\end{equation}
\end{remark}
Next we define the following quantities, that are closely related to
${\hat A}_{u,q}$:
\begin{definition}\label{gjghghghjgghGHKKzzbvqKKkk}
Given a compact set $ \ov U\subset\subset\O$ let
\begin{equation}
\begin{aligned}
B_{u,q}\big(\ov U\big):=\limsup\limits_{\e\to 0^+}\sup\limits_{\vec
k\in S^{N-1}}\int_{\ov U}\frac{1}{\e}\Big|u(x+\e\vec
k)-u(x)\Big|^qdx \label{GMT'1jGHKKzzbvq}.
\end{aligned}
\end{equation}
%\text{ and }
%\end{align}
%\begin{align}
Next, given an open set $\Omega\subset\R^N$ define
\begin{equation}
\label{hjhjhghgjkghggghGHKKjjzzbvqjj}
B_{u,q}\big(\Omega\big):=\sup\limits_{K\subset\subset\Omega}B_{u,q}\big(K\big).
\end{equation}
Finally, set
\begin{equation}
\label{GMT'3jGHKKkkhjjhgzzZZzzbvqkk} \hat
B_{u,q}\big(\R^N\big):=\limsup\limits_{\e\to 0^+}\sup\limits_{\vec
k\in S^{N-1}}\int_{\R^N}\frac{1}{\e}\Big|u(x+\e\vec
k)-u(x)\Big|^qdx.
%\\=\sup\limits_{\{h\in\R^N:\,0<|h|< dist(K,\R^N\setminus\Omega)\}}\Bigg(\int\limits_K\frac{1}{|h|}\Big|u(x+h)-u(x)\Big|^qdx\Bigg)
\end{equation}
\end{definition}
The following result is known; for the
convenience of a reader we will give its proof in the
Appendix.
\begin{lemma}\label{hjgjg}
For any $q>1$, a function $u\in L^q(\R^N,\R^d)$ belongs to
$B_{q,\infty}^{1/q}(\R^N,\R^d)$ if and only if $\hat
B_{u,q}\big(\R^N\big)<\infty$. Moreover, for any  open
$\Omega\subset\R^N$, a function $u\in L^q_{loc}(\Omega,\R^d)$ belongs
to $\big(B_{q,\infty}^{1/q}\big)_{loc}(\Omega,\R^d)$ if and only
if for every compact $K\subset\subset\Omega$ we have
$B_{u,q}\big(K\big)<\infty$.
\end{lemma}
Then Proposition \ref{ghgghgghzzbvqhjhjh} is a part of the following
statment:
\begin{proposition}\label{ghgghgghzzbvq}
For every open set $\Omega\subset\R^N$, every $q\geq 1$ and $u\in
L^q_{loc}(\Omega,\R^d)$ we have
\begin{equation}\label{hjhjhghgjkghggghGHGHGHKKkljghhgzzbvqaa}
\frac{A_{u,q}(\Omega)}{\mathcal{L}^N({B_1(0)})}\leq
B_{u,q}\big(\Omega\big)\leq
2^{N+q}\frac{A_{u,q}(\Omega)}{\mathcal{L}^N({B_1(0)})}\,.
\end{equation}
Moreover, if $u\in
L^q(\R^N,\R^d)$ then
\begin{equation}\label{hjhjhghgjkghggghGHGHGHKKkljghhgzzbvqkkkjkhjhgjg}
%\frac{A_{u,q}(\R^N)}{\mathcal{L}^N({B_1(0)})}=
\frac{\hat A_{u,q}(\R^N)}{\mathcal{L}^N({B_1(0)})}\leq \hat
B_{u,q}\big(\R^N\big)\leq 2^{N+q}\frac{\hat
A_{u,q}(\R^N)}{\mathcal{L}^N({B_1(0)})}.
%=2^{N+q}\frac{A_{u,q}(\R^N)}{\mathcal{L}^N({B_1(0)})}.
\end{equation}
In particular, for $q>1$ we have:
\begin{equation}\label{hjhjjggjbnvnv}
BV^q(\R^N,\R^d)=B_{q,\infty}^{1/q
%\frac{1}{q}
}(\R^N,\R^d)\quad\quad\text{and}\quad\quad
BV^q_{loc}(\Omega,\R^d)=\big(B_{q,\infty}^{ 1/q
%\frac{1}{q}
}\big)_{loc}(\Omega,\R^d).
\end{equation}
%%In particular, $BV^q(\R^N,\R^d)=B_{q,\infty}^{1/q
%\frac{1}{q}
%%}(\R^N,\R^d)$ and $BV^q_{loc}(\Omega,\R^d)=\big(B_{q,\infty}^{ 1/q
%\frac{1}{q}
%%}\big)_{loc}(\Omega,\R^d)$ $\;\forall q>1$.
%$\hat{BV}^q(\Omega,\R^d)=BV^q(\Omega,\R^d)$.
\end{proposition}
Proposition~\ref{ghgghgghzzbvq} will be deduced from
Lemma~\ref{hjhhjKKzzbvq} below.

The next  theorem is a generalization of Theorem
\ref{ghgghgghjjkjkzzbvqred}:
\begin{theorem}\label{ghgghgghjjkjkzzbvq}
Let $\Omega\subset\R^N$ be an open set and let $u\in
BV_{loc}(\Omega,\R^d)\cap L^\infty_{loc}(\Omega,\R^d)$. Then, for
every $q>1$ we have $u\in {BV}^q_{loc}(\Omega,\R^d)$ and for every
compact set $K\subset\subset\Omega$
%with Lipschitz boundary
such that $\|Du\|(\partial K)=0$ we have
\begin{equation}\label{fgyufghfghjgghgjkhkkGHGHKKjjjjkjkkjkmmlmjijilzzbvqKK}
A_{u,q}\big(K\big)=C_N\int_{J_u\cap
K}\Big|u^+(x)-u^-(x)\Big|^qd\mathcal{H}^{N-1}(x),
\end{equation}
where $C_N$ is defined in
\er{fgyufghfghjgghgjkhkkGHGHKKggkhhjoozzbvqkk}. Moreover, if in
addition $u\in BV(\Omega,\R^d)\cap L^\infty(\Omega,\R^d)$, then for
every $q>1$ we have
%$u(x)\in ({BV}^q)'(\ov\Omega,\R^d)$ and
\begin{equation}\label{fgyufghfghjgghgjkhkkGHGHKKjjjjkjkkjkmmlmjijiluuizzbvqKK}
A_{u,q}\big(\Omega\big)=C_N\int_{J_u}\Big|u^+(x)-u^-(x)\Big|^qd\mathcal{H}^{N-1}(x).
\end{equation}
Finally, if $\Omega$ is an open set with a bounded Lipschitz boundary
and $u\in BV(\Omega,\R^d)\cap L^\infty(\Omega,\R^d)$ then we have
$u\in {BV}^q(\Omega,\R^d)$ for every $q>1$ and
\begin{equation}\label{fgyufghfghjgghgjkhkkGHGHKKjjjjkjkkjkmmlmjijiluuizziihhhjbvqKK}
\hat
A_{u,q}\big(\Omega\big)=C_N\int_{J_u}\Big|u^+(x)-u^-(x)\Big|^qd\mathcal{H}^{N-1}(x)=A_{u,q}\big(\Omega\big).
\end{equation}
\end{theorem}
%
%
%
\begin{comment}
\begin{theorem}\label{ghgghgghjjkjkzzbvq}
Let $\Omega\subset\R^N$ be an open set with bounded Lipschitz
boundary and let $u\in BV(\Omega,\R^d)\cap L^\infty(\Omega,\R^d)$.
Then, for every $q>1$ we have $u\in {BV}^q(\Omega,\R^d)$ and
\begin{equation}\label{fgyufghfghjgghgjkhkkGHGHKKjjjjkjkkjkmmlmjijiluuizziihhhjbvq}
\hat
A_{u,q}\big(\Omega\big)=C_N\int_{J_u}\Big|u^+(x)-u^-(x)\Big|^qd\mathcal{H}^{N-1}(x),
%=A_{u,q}\big(\Omega\big),
\end{equation}
with the dimensional constant $C_N>0$ defined by
\begin{equation}\label{fgyufghfghjgghgjkhkkGHGHKKggkhhjoozzbvq}
C_N:=\frac{1}{N}\int_{S^{N-1}}|z_1|d\mathcal{H}^{N-1}(z)\,,
\end{equation}
where we denote $z:=(z_1,\ldots, z_N)\in\R^N$.
\end{theorem}
\end{comment}
%
%
%
%As we shall see in Section~\ref{hjhjghjyh},
%Theorem~\ref{ghgghgghjjkjkzzbvq} can be deduced from a more general
%result, Theorem~\ref{ghgghgghjjkjkzzbvqKK}.

%
%
%
%

%
%
%
The next proposition
%~\ref{ghgfgc}
is an easy consequence of the definitions; the details are left to
the reader.
\begin{proposition}\label{ghgfgc}
For every open set $\Omega\subset\R^N$, two real numbers $q_2> q_1\geq
1$ and  $u\in L^\infty(\Omega,\R^d)$ we have
\begin{multline}\label{GMT'3jGHKKkkhjjhgzzZZzzZZzzbvqhjjhij}
\bar A_{u,q_2}\big(\Omega\big)\leq
2^{q_2-q_1}\|u\|^{q_2-q_1}_{L^\infty(\Omega,\R^d)}\bar
A_{u,q_1}\big(\Omega\big),\quad\hat A_{u,q_2}\big(\Omega\big)\leq
2^{q_2-q_1}\|u\|^{q_2-q_1}_{L^\infty(\Omega,\R^d)}\hat
A_{u,q_1}\big(\Omega\big),\\ \quad A_{u,q_2}\big(\Omega\big)\leq
2^{q_2-q_1} \|u\|^{q_2-q_1}_{L^\infty(\Omega,\R^d)}
A_{u,q_1}\big(\Omega\big)\quad\text{and}\quad
B_{u,q_2}\big(\Omega\big)\leq
2^{q_2-q_1}\|u\|^{q_2-q_1}_{L^\infty(\Omega,\R^d)}
B_{u,q_1}\big(\Omega\big).
\end{multline}
In particular, for every open set $\Omega\subset\R^N$ and any two
real numbers $q_2> q_1\geq 1$  we have
$BV^{q_1}_{loc}(\Omega,\R^d)\cap L^\infty_{loc}(\Omega,\R^d)\subset
BV^{q_2}_{loc}(\Omega,\R^d)$ and $BV^{q_1}(\Omega,\R^d)\cap
L^\infty(\Omega,\R^d)\subset BV^{q_2}(\Omega,\R^d)$.
\end{proposition}
\begin{remark}\label{bhhfgcczzbvqbnbnhjh}
If $\Omega\subset\R^N$ is an open set, $D\subset\R^N$ is a Borel set
and $\chi_D$ is the characteristic function of $D$, i.e.,
\begin{equation}\label{fgyuhjjhhh}
\chi_D(x):=\begin{cases}1 & x\in
D,\\
0& x\notin D,
\end{cases}
\end{equation}
then clearly for every $q\geq 1$ we have:
\begin{equation}\label{GMT'3jGHKKkkhjjhgzzZZzzZZzzbvqhjjhijjjj}
\begin{aligned}
\bar A_{\chi_D,q}\big(\Omega\big)&=\bar
A_{\chi_D,1}\big(\Omega\big),\quad\hat
A_{\chi_D,q}\big(\Omega\big)&=\hat A_{\chi_D,1}\big(\Omega\big),\\
A_{\chi_D,q}\big(\Omega\big)&= A_{\chi_D,1}\big(\Omega\big),\quad
B_{\chi_D,q}\big(\Omega\big)&=B_{\chi_D,1}\big(\Omega\big).
\end{aligned}
\end{equation}
In particular, $\chi_D\in BV^q_{loc}(\Omega,\R^d)$ if and only if
$D$ has a locally finite perimeter. Moreover, if in addition
$\mathcal{L}^N(D)<\infty$ then we have $\chi_D\in BV^q(\Omega,\R^d)$
if and only if $D$ has finite perimeter.
\end{remark}
%
%
%
\begin{comment}
\begin{remark}
Similar results hold also for more general mollifiers than in
\eqref{hgghg}, e.g., of the form
$\rho_\e(x)=\frac{\rho(|x|/\e)}{\e^N}$, where $\rho(t)$ is a
nonnegative function on $[0,\infty)$ with compact support, such that
$\essinf_{(0,\delta)} \rho\geq \alpha$ for some $\alpha,\delta>0$
and  $\int_{\R^N}\rho(|x|)\,dx=1$. We did not investigate more
general families of radial mollifiers $\{\rho_\e(x)\}$ as in
\cite{hhh1}.
\end{remark}
\end{comment}
%
%
%

In the special case $N=1$, i.e., when the domain $\Omega$ is an
interval, there exists a classical notion of a space  of
functions of bounded $q$-variation  (see e.g., Kolyada and
Lind~\cite{kolyada} and the references therein). This space, denoted
by
$V_q\big(\Omega,\R^d\big)$, was first considered by
Wiener~\cite{wiener} (for $q=2$). Below we recall the
definition of $V_q\big(\Omega,\R^d\big)$ and also define its
a.e.-equivalent version that we denote by
$\hat V_q\big(\Omega,\R^d\big)$.
\begin{definition}\label{hghjghghgh}
Given an interval $I\subseteq\R$ (open, closed, bounded or
unbounded) denote for every $n\in \mathbb{N}$,
\begin{equation*}
\Pi_n(I):=\Big\{(x_1,x_2,\ldots,x_{n+1})\in\R^{n+1}\,:\;
x_1<x_2<\ldots<x_n<x_{n+1},\;\;x_1\in I,\;\;x_{n+1}\in I\Big\}.
\end{equation*}
For any function $f:I\to\R^d$ defined {\em everywhere} in $I$  and
for every $q\geq 1$ let
\begin{equation}\label{vhghfggfgfgfdfdiuuiui}
v_{q,I}(f):=\sup\limits_{n\in\mathbb{N}}\Bigg(\sup\limits_{(x_1,\ldots,x_{n+1})\in\Pi_n(I)}\Big(\sum_{k=1}^{n}\big|f(x_{k+1})-f(x_k)\big|^q\Big)^{\frac{1}{q}}\Bigg)\,.
\end{equation}
We shall say that $f\in V_q(I,\R^d)$
%has a bounded $q$-variation in $I$
if $v_{q,I}(f)<\infty$. Next, for a measurable $\R^d$-valued
function $f$, defined a.e. in $I$, and $q\geq 1$ let
\begin{equation}\label{vhghfggfgfgfdfdiuuiuihjhj}
\hat
v_{q,I}(f):=\inf\bigg\{v_{q,I}(g)\,:\;g:I\to\R^d,\;g(x)=f(x)\;\,\text{a.e.
in}\;\,I\bigg\}.
\end{equation}
We shall say that such $f$
%has an essential bounded $q$-variation in $I$
belongs to the space $\hat V_q(I,\R^d)$ if $\hat v_{q,I}(f)<\infty$.
Evidently, if $\hat v_{q,I}(f)<\infty$ then $f\in
L^\infty\big(I,\R^d\big)$ and moreover, if $v_{q,I}(f)<\infty$ then
$f$ is bounded everywhere.
\end{definition}
The next Proposition is concerned with the relation between the spaces $\hat
V_q\big([a,b],\R^d\big)$ and $BV^q\big((a,b),\R^d\big)$:
\begin{proposition}\label{gygfhffghlkoii}
For every $q\geq 1$ and every $a<b\in\mathbb{R}$, if a measurable
function $f:(a,b)\to\R^d$ defined a.e.~in $(a,b)$
%has an essential bounded $q$-variation in $[a,b]$
belongs to the space $\hat V_q\big([a,b],\R^d\big)$, then $f\in
BV^q\big((a,b),\R^d\big)$. Moreover, we have:
\begin{equation}\label{vhghfggfgfgfdfdiuuiuihjhjlklkouiui}
\bar A_{f,q}\big((a,b)\big)\leq 4\big(\hat v_{q,[a,b]}(f)\big)^q.
\end{equation}
I.e. the space $\hat V_q\big([a,b],\R^d\big)$ is continuously
embedded in $BV^q\big((a,b),\R^d\big)$.
\end{proposition}
The proof of Proposition~\ref{gygfhffghlkoii} is given  in the
Appendix.
\begin{remark}\label{ghjhghhhhfff}
By Proposition~\ref{gygfhffghlkoii} we have $\hat
V_q\big([a,b],\R^d\big)\subset BV^q\big((a,b),\R^d\big)$. While for
$q=1$ it is well known that the two spaces coincide, the inclusion
is {\em strict} when $q>1$. Indeed, while $\hat
V_q\big([a,b],\R^d\big)\subset L^\infty\big((a,b),\R^d\big)$, by
Theorem~\ref{huyhuyuughhjhhjjhhjjjhhjhjjhhhjzzbvq} we have
$W^{\frac{1}{q},q}\big((a,b),\R^d\big)\subset
BV^q\big((a,b),\R^d\big)$ and it is well known that for $q>1$,
$W^{\frac{1}{q},q}\big((a,b),\R^d\big)\setminus
L^\infty\big((a,b),\R^d\big)\neq\emptyset$.
\end{remark}

\subsection{Proofs of the main results for the space
$BV^q$}\label{hjhjghjyh}
 We begin with two technical Lemmas that are used in the proof of Proposition~\ref{ghgghgghzzbvq}.
\begin{lemma}\label{gughfgfhfgdgddffddfKKzzbvq}
Let $\Omega\subset\R^N$ be an open set, $q\geq 1$  and let $u\in
L^q_{loc}(\Omega,\R^d)$. Then,  for every open
$\Omega_1\subset\subset \Omega_2\subset\subset\Omega$, for every
$h_1\in\R^N$ such that $0<|h_1|\leq
\dist(\Omega_1,\R^N\setminus\Omega_2)$ and every $h_2\in\R^N$ such
that $0<|h_2|\leq \dist(\Omega_2,\R^N\setminus\Omega)$, we have
\begin{multline}\label{gghgjhfgggjfgfhughGHGHGHjgghKKzzbvq}
\int_{\ov\Omega_1}\frac{1}{|h_1+h_2|}\Big|u\big(x+(h_1+h_2)\big)-u(x)\Big|^qdx
\leq\\
2^{q-1}\Bigg(\frac{|h_2|}{|h_1+h_2|}\int_{\ov\Omega_2}\frac{1}{|h_2|}\Big|u(x+h_2)-u(x)\Big|^qdx+\frac{|h_1|}{|h_1+h_2|}\int_{\ov\Omega_1}
\frac{1}{|h_1|}\Big|u(x+h_1)-u(x)\Big|^qdx\Bigg).
\end{multline}
In particular, for every $h\in S^{N-1}$, $0<\e_1\leq
\dist(\Omega_1,\R^N\setminus\Omega_2)$ and  $0<\e_2\leq
\dist(\Omega_2,\R^N\setminus\Omega)$, we have
\begin{multline}
\label{gghgjhfgggjfgfhughGHGHGHjgghuggghghihhjhjjhKKzzbvq}
\int_{\ov\Omega_1}\frac{1}{\e_1+\e_2}\Big|u\big(x+(\e_1+\e_2)h\big)-u(x)\Big|^qdx
\leq\\
2^{q-1}\max\Bigg\{\int_{\ov\Omega_2}\frac{1}{\e_2}\Big|u(x+\e_2
h)-u(x)\Big|^qdx\,,\,\int_{\ov\Omega_1}
\frac{1}{\e_1}\Big|u(x+\e_1h)-u(x)\Big|^qdx\Bigg\}.
\end{multline}
\end{lemma}
\begin{proof}
By the triangle inequality and the convexity
of $g(s):=|s|^q$ we have
\begin{multline*}
%\label{gghgjhfgggjfgfhughGHGHGHKK}
\int_{\ov\Omega_1}\frac{1}{|h_1+h_2|}\Big|u\big(x+(h_1+h_2)\big)-u(x)\Big|^qdx=\\ \int_{\ov\Omega_1}\frac{1}{|h_1+h_2|}\Big|u\big(x+(h_1+h_2)\big)-u(x+h_1)+u(x+h_1)-u(x)\Big|^qdx\leq\\
\int_{\ov\Omega_1}\frac{1}{|h_1+h_2|}\bigg(\Big|u\big(x+(h_1+h_2)\big)-u(x+h_1)\Big|+\Big|u(x+h_1)-u(x)\Big|\bigg)^qdx\leq\\
\int_{\ov\Omega_1}\frac{2^{q-1}}{{|h_1+h_2|}}\Big|u\big(x+(h_1+h_2)\big)-u(x+h_1)\Big|^qdx+\int_{\ov\Omega_1}\frac{2^{q-1}}{|h_1+h_2|}\Big|u(x+h_1)-u(x)\Big|^qdx=\\
\frac{2^{q-1}|h_2|}{|h_1+h_2|}\int_{\ov\Omega_1}\frac{1}{|h_2|}\Big|u\big((x+h_1)+h_2\big)-u(x+h_1)\Big|^qdx+\frac{2^{q-1}|h_1|}{|h_1+h_2|}\int_{\ov\Omega_1}
\frac{1}{|h_1|}\Big|u(x+h_1)-u(x)\Big|^qdx\\
\leq
2^{q-1}\Bigg(\frac{|h_2|}{|h_1+h_2|}\int_{\ov\Omega_2}\frac{1}{|h_2|}\Big|u(x+h_2)-u(x)\Big|^qdx+\frac{|h_1|}{|h_1+h_2|}\int_{\ov\Omega_1}
\frac{1}{|h_1|}\Big|u(x+h_1)-u(x)\Big|^qdx\Bigg).
\end{multline*}
In particular, for every $h\in S^{N-1}$, $0<\e_1\leq
\dist(\Omega_1,\R^N\setminus\Omega_2)$ and  $0<\e_2\leq
\dist(\Omega_2,\R^N\setminus\Omega)$ we have
\begin{multline*}
%\label{gghgjhfgggjfgfhughGHGHGHjgghuggghghKK}
\int_{\ov\Omega_1}\frac{1}{\e_1+\e_2}\Big|u\big(x+(\e_1+\e_2)h\big)-u(x)\Big|^qdx
\leq\\
2^{q-1}\Bigg(\frac{\e_2}{\e_1+\e_2}\int_{\ov\Omega_2}\frac{1}{\e_2}\Big|u(x+\e_2
h)-u(x)\Big|^qdx+\frac{\e_1}{\e_1+\e_2}\int_{\ov\Omega_1}
\frac{1}{\e_1}\Big|u(x+\e_1h)-u(x)\Big|^qdx\Bigg)\\ \leq
2^{q-1}\max\Bigg\{\int_{\ov\Omega_2}\frac{1}{\e_2}\Big|u(x+\e_2
h)-u(x)\Big|^qdx\,,\,\int_{\ov\Omega_1}
\frac{1}{\e_1}\Big|u(x+\e_1h)-u(x)\Big|^qdx\Bigg\}.
\end{multline*}
\end{proof}
\begin{lemma}\label{hjhhjKKzzbvq}
Let $\Omega\subset\R^N$ be an open set, $q\geq 1$ and let   $u\in
L^q_{loc}(\Omega,\R^d)$. Then,  for every open
$\Omega_1\subset\subset \Omega_2\subset\subset\Omega$, $\vec k\in
S^{N-1}$ and $\e$ satisfying
\begin{equation}
\label{eq:epsil}
0<\e\leq \min{\big\{\dist(\Omega_1,\R^N\setminus\Omega_2),\dist(\Omega_2,\R^N\setminus\Omega)\big\}}\,,
\end{equation}
we have
\begin{equation}\label{gghgjhfgggjfgfhughGHGHKKzzjkjkyuyuybvq}
\int_{\ov\Omega_1}\frac{1}{\e}\Big|u(x+\e\vec k)-u(x)\Big|^qdx \leq
\frac{2^{N+q}}{\mathcal{L}^N({B_1(0)})}\int_{{B_1(0)}}\int_{\ov\Omega_2}\frac{1}{\e|z|}\Big|u(x+\e
z)-u\big(x)\Big|^qdxdz.
\end{equation}
Moreover, if $u\in L^q(\R^N,\R^d)$ then
\begin{equation}\label{gghgjhfgggjfgfhughGHGHKKzzjkjkyuyuybvqkkkk}
\int_{\R^N}\frac{1}{\e}\Big|u(x+\e\vec k)-u(x)\Big|^qdx \leq
\frac{2^{N+q}}{\mathcal{L}^N({B_1(0)})}\int_{{B_1(0)}}\int_{\R^N}\frac{1}{\e|z|}\Big|u(x+\e
z)-u\big(x)\Big|^qdxdz.
\end{equation}
In particular,
\begin{equation}\label{hjhjhghgjkghggghGHGHGHKKkljghhgzzbvq}
\frac{A_{u,q}(\ov\Omega_1)}{\mathcal{L}^N({B_1(0)})}\leq
B_{u,q}\big(\ov\Omega_1\big),
\end{equation}
and
\begin{equation}\label{hjhjhghgjkghggghGHGHGHKKzzbvq}
B_{u,q}\big(\ov\Omega_1\big)\leq
2^{N+q}\frac{A_{u,q}(\ov\Omega_2)}{\mathcal{L}^N({B_1(0)})}
\end{equation}
(see Definitions \ref{gjghghghjgghGHKKzzbvqKK} and
\ref{gjghghghjgghGHKKzzbvqKKkk}). Moreover, if $u\in L^q(\R^N,\R^d)$
then
\begin{equation}\label{hjhjhghgjkghggghGHGHGHKKkljghhgzzbvqkkk}
%\frac{A_{u,q}(\R^N)}{\mathcal{L}^N({B_1(0)})}=
\frac{\hat A_{u,q}(\R^N)}{\mathcal{L}^N({B_1(0)})}\leq \hat
B_{u,q}\big(\R^N\big),
\end{equation}
and
\begin{equation}\label{hjhjhghgjkghggghGHGHGHKKzzbvqkkk}
\hat B_{u,q}\big(\R^N\big)\leq 2^{N+q}\frac{\hat
A_{u,q}(\R^N)}{\mathcal{L}^N({B_1(0)})}.
%=2^{N+q}\frac{A_{u,q}(\R^N)}{\mathcal{L}^N({B_1(0)})}.
\end{equation}

\end{lemma}
\begin{proof}
Inequalities \er{hjhjhghgjkghggghGHGHGHKKkljghhgzzbvq} and
\er{hjhjhghgjkghggghGHGHGHKKkljghhgzzbvqkkk} are clear from the
definitions. Next, by \er{gghgjhfgggjfgfhughGHGHGHjgghKKzzbvq} we
have: for every $\vec k\in \R^N\setminus\{0\}$ such that $|\vec
k|\leq 1$, for every $z\in\R^N$ such that $\big|z-\frac{1}{2}\vec
k\big|<\frac{1}{2}|\vec k|$, $\Omega_1\subset\subset
\Omega_2\subset\subset\Omega$ and $\e$ satisfying \eqref{eq:epsil}
there holds
\begin{multline}\label{gghgjhfgggjfgfhughGHKKzzbvq}
\int_{\ov\Omega_1}\frac{1}{\e|\vec k|}\Big|u(x+\e\vec
k)-u(x)\Big|^qdx \leq\\ \frac{2^{q-1}|z|}{|\vec
k|}\int_{\ov\Omega_2}\frac{1}{\e| z|}\Big|u(x+\e
z)-u\big(x)\Big|^qdx+\frac{2^{q-1}|\vec k-z|}{|\vec
k|}\int_{\ov\Omega_1}\frac{1}{\e|\vec k-z|}\Big|u\big(x+\e(\vec
k-z)\big)-u(x)\Big|^qdx\,.
\end{multline}
Since the inequality $\big|z-\frac{1}{2}\vec
k\big|<\frac{1}{2}|\vec k|$ implies the inequalities $|z|<|\vec k|$ and
$|\vec k-z|<|\vec k|$, we have by \er{gghgjhfgggjfgfhughGHKKzzbvq}:
\begin{multline}\label{gghgjhfgggjfgfhughGHGHKKzzbvq}
\int_{\ov\Omega_1}\frac{1}{\e|\vec k|}\Big|u(x+\e\vec
k)-u(x)\Big|^qdx \leq\\
\frac{2^{q-1}}{\mathcal{L}^N(\{z\in\R^N:\,\big|z-\frac{1}{2}\vec
k\big|<\frac{1}{2}|\vec
k|\})}\int_{\{z\in\R^N:\,\big|z-\frac{1}{2}\vec
k\big|<\frac{1}{2}|\vec k|\}}\int_{\ov\Omega_2}\frac{1}{\e|
z|}\Big|u(x+\e
z)-u\big(x)\Big|^qdxdz\\+\frac{2^{q-1}}{\mathcal{L}^N(\{z\in\R^N:\,\big|z-\frac{1}{2}\vec
k\big|<\frac{1}{2}|\vec
k|\})}\int_{\{z\in\R^N:\,\big|z-\frac{1}{2}\vec
k\big|<\frac{1}{2}|\vec
k|\}}\int_{\ov\Omega_1}\frac{\Big|u\big(x+\e(\vec
k-z)\big)-u(x)\Big|^q}{\e|\vec
k-z|}dxdz\\
\leq\frac{2^{N+q-1}}{|\vec
k|^N\mathcal{L}^N({B_1(0)})}\int_{\{z\in\R^N:\,|z|<|\vec
k|\}}\int_{\ov\Omega_2}\frac{1}{\e| z|}\Big|u(x+\e
z)-u\big(x)\Big|^qdxdz\\+\frac{2^{N+q-1}}{|\vec
k|^N\mathcal{L}^N({B_1(0)})}\int_{\{z\in\R^N:\,|\vec k-z|<|\vec
k|\}}\int_{\ov\Omega_1}\frac{1}{\e|\vec k-z|}\Big|u\big(x+\e (\vec k-z)\big)-u(x)\Big|^qdxdz=\\
\frac{2^{N+q}}{|\vec
k|^N\mathcal{L}^N({B_1(0)})}\int_{\{z\in\R^N:\,|z|<|\vec
k|\}}\int_{\ov\Omega_2}\frac{1}{\e|
z|}\Big|u(x+\e z)-u\big(x)\Big|^qdxdz=\\
\frac{2^{N+q}}{\mathcal{L}^N({B_1(0)})}\int_{{B_1(0)}}\int_{\ov\Omega_2}\frac{1}{\e|\vec
k||z|}\Big|u(x+\e|\vec k|z)-u\big(x)\Big|^qdxdz\,,
\end{multline}
and \er{gghgjhfgggjfgfhughGHGHKKzzjkjkyuyuybvq} follows.
%\begin{equation}\label{gghgjhfgggjfgfhughGHGHKKzzjkjkyuyuy}
%\int_{\ov\Omega_1}\frac{1}{\e}\Big|u(x+\e\vec k)-u(x)\Big|^qdx \leq\frac{2^{N+q}}{\mathcal{L}^N({B_1(0)})}\int_{{B_1(0)}}\int_{\ov\Omega_2}\frac{1}{\e|z|}\Big|u(x+\e z)-u\big(x)\Big|^qdxdz.
%\end{equation}
Moreover, if $u\in L^q(\R^N,\R^d)$ then taking the supremum of
\er{gghgjhfgggjfgfhughGHGHKKzzjkjkyuyuybvq} over all
$\O_1\subset\subset\Omega_2\subset\subset\R^N$ we deduce
\er{gghgjhfgggjfgfhughGHGHKKzzjkjkyuyuybvqkkkk}.

Finally, from \er{gghgjhfgggjfgfhughGHGHKKzzbvq} we deduce that for every $0<\rho\leq
1$ we have
\begin{multline}\label{gghgjhfgggjfgfhughGHGHyuuyGHKKzzbvq}
\limsup\limits_{\e\to 0^+}\Bigg(\sup\limits_{\{\vec k\in\R^N:\,|\vec
k|=\rho\}}\bigg(\int_{\ov\Omega_1}\frac{1}{\e\rho}\Big|u(x+\e\vec
k)-u(x)\Big|^qdx\bigg)\Bigg) \leq\\
\frac{2^{N+q}}{\mathcal{L}^N({B_1(0)})}\limsup\limits_{\e\to
0^+}\Bigg(\int_{{B_1(0)}}\int_{\ov\Omega_2}\frac{1}{\e\rho|z|}\Big|u(x+\e\rho
z)-u\big(x)\Big|^qdxdz\Bigg)\,,
\end{multline}
which clearly implies \er{hjhjhghgjkghggghGHGHGHKKzzbvq}. Moreover,
if $u\in L^q(\R^N,\R^d)$ then by
\er{gghgjhfgggjfgfhughGHGHKKzzjkjkyuyuybvqkkkk} we have:
\begin{multline}\label{gghgjhfgggjfgfhughGHGHKKzzjkjkyuyuybvqkkkkjjkkk}
\limsup\limits_{\e\to 0^+}\Bigg(\sup\limits_{\{\vec k\in\R^N:\,|\vec
k|=1\}}\bigg(\int_{\R^N}\frac{1}{\e}\Big|u(x+\e\vec
k)-u(x)\Big|^qdx\bigg)\Bigg)\\ \leq
\frac{2^{N+q}}{\mathcal{L}^N({B_1(0)})}\limsup\limits_{\e\to
0^+}\Bigg(\int_{{B_1(0)}}\int_{\R^N}\frac{1}{\e|z|}\Big|u(x+\e
z)-u\big(x)\Big|^qdxdz\Bigg),
\end{multline}
which clearly implies \er{hjhjhghgjkghggghGHGHGHKKzzbvqkkk}.
\end{proof}
 The next Proposition is a key ingredient in the proof of
 Theorem~\ref{ghgghgghjjkjkzzbvq}.
\begin{proposition}\label{hgugghghhffhfhKKzzbvq}
Let $\Omega\subset\R^N$ be an open set and let
$W(a,b):\R^d\times\R^d\to\R$ be a nonnegative continuously
differentiable function,
%differentiable w.r.t~the $a$-variable,
which satisfies $W(a,a)=0$ and $W(a,b)=W(b,a)$ for every
$a,b\in\R^d$.
% and $x\in \R^{N}$
% and such that $\nabla_a W(a,b)$ is a continuous mapping.
Let $u\in BV_{loc}(\Omega,\R^d)\cap L^\infty_{loc}(\Omega,\R^d)$.
Then, for every compact set $K\subset\subset\Omega$
%with Lipschitz boundary
such that $\|Du\|(\partial K)=0$ and any vector $\vec
k\subset\R^N$ we have
\begin{equation}\label{fgyufghfghjgghgjkhkkGHKKzzbvq}
\lim_{t\to0^+}\frac{1}{t}\int_KW\Big(u(x+t\vec
k),u(x)\Big)dx=\int_{J_u\cap K}W\Big(u^+(x),u^-(x)\Big)\big|\vec
k\cdot\vec\nu(x)\big|d\mathcal{H}^{N-1}(x).
\end{equation}
In particular, for $q>1$ we have
\begin{equation}\label{fgyufghfghjgghgjkhkkGHGHKKzzbvq}
%D_{u,q}\big(K,\vec k\big)=
\lim_{t\to0^+}\frac{1}{t}\int_K\Big|u(x+t\vec
k)-u(x)\Big|^qdx=\int_{J_u\cap
K}\Big|u^+(x)-u^-(x)\Big|^q\big|\vec
k\cdot\vec\nu(x)\big|d\mathcal{H}^{N-1}(x),
\end{equation}
and
\begin{multline}\label{fgyufghfghjgghgjkhkkGHGHKKjjjjkjkkjzzbvq}
A_{u,q}\big(K\big)=\lim\limits_{t\to
0^+}\int_{{B_1(0)}}\int_K\frac{1}{t|z|}\Big|u(x+t
z)-u(x)\Big|^qdxdz\\=C_N\bigg(\int_{J_u\cap
K}\Big|u^+(x)-u^-(x)\Big|^qd\mathcal{H}^{N-1}(x)\bigg),
\end{multline}
with $C_N$ defined in
\er{fgyufghfghjgghgjkhkkGHGHKKggkhhjoozzbvqkk}.

\end{proposition}
\begin{proof}
Let  $\eta(z)\in C^\infty_c(\R^N,\R)$ be a radial function such that
$\eta\geq 0$, $supp\,\eta\subset B_1(0)$ and
$\int_{\R^N}\eta(z)dz=1$. For every $x\in\Omega$ and every
$0<\e<\dist(x,\partial\Omega)$ define
\begin{equation}\label{jmvnvnbccbvhjhjhhjjkhgjgGHKKzzbvq}
u_\e(x):=\frac{1}{\e^N}\int_{\R^N}\eta\Big(\frac{y-x}{\e}\Big)u(y)dy=\int_{\R^N}\eta(z)u(x+\e
z)dz=\int_{B_1(0)}\eta(z)u(x+\e z)dz.
\end{equation}
Then, following definition \ref{defjac889878}, we have
\begin{equation}\label{jmvnvnbccbvhjhjhhjjkhgjgGHKKhugjzzbvq}
\lim_{\e\to0^+}u_\e(x)=\bar u(x):=\begin{cases} \tilde u(x)\quad
x\in\Omega\setminus J_u\\ \frac{1}{2}\big(u^+(x)+u^-(x)\big)\quad
x\in J_u\end{cases}\quad\quad\text{for}\;\;
\mathcal{H}^{N-1}{-a.e.}\;\;x\in\Omega.
\end{equation}
Moreover, since there exist two open sets $U_1\subset\subset
U_2\subset\subset\Omega$ such that $K\subset\subset U_1$ and $u\in
L^\infty(U_2,\R^d)$, we deduce that there exist constants $M>0$ and
$\e_0>0$, such that
\begin{equation}\label{jmvnvnbccbvhjhjhhjjkhgjgGHKKhugjzzbvqihhjhjkhkll}
\begin{cases}
\big|u_\e(x)\big|\leq M\quad\quad\forall x\,\in U_1,\;\forall\,
\e\in(0,\e_0),\\
%\quad\text{and}\quad
\big|\bar u(x)\big|\leq M\quad\quad\text{for}\;\;
\mathcal{H}^{N-1}{-a.e.}\;\;x\in U_1.
\end{cases}
\end{equation}
%Let us denote for short the L.H.S.~of \eqref{fgyufghfghjgghgjkhkkGHKKzzbvq} by $I$.
Then, denoting for any $s\in[0,1]$, $x\in \Omega$ and $\vec
k\subset\R^N$
\begin{equation*}
P_t(u_\e,x,s,\vec k)=su_\e(x+t\vec k)+(1-s)u_\e(x)\,,
\end{equation*}
using the Dominated Convergence Theorem, the Fundamental
Theorem of Calculus and finally
\er{jmvnvnbccbvhjhjhhjjkhgjgGHKKzzbvq}, we get for small $t>0$,
%three compacts $K_0\subset\subset K\subset\subset\bar
%K\subset\subset\Omega$
%such that $\dist(K_0,\Omega\setminus K)>0$ and $\dist(K,\Omega\setminus \bar K)>0$ for every $t>0$ such that%$$0<t\leq
%\frac{1}{2}\min\Big\{dist(K_0,\Omega\setminus K),dist(K,\Omega\setminus \bar K)\Big\}$$ we have:
\begin{multline}\label{fgyufghfghGHKK1zzbvqgfgjh}
I_t:=\frac{1}{t}\int_KW\Big(u(x+t\vec
k),u(x)\Big)dx=\frac{1}{t}\lim_{\e\to0+}\int_K W\Big(u_\e(x+t\vec
k),u_\e(x)\Big)dx=\\ \frac{1}{t}\lim_{\e\to0+}\int_K\int_0^1
\nabla_a W\Big( P_t(u_\e,x,s,\vec k),u_\e(x)\Big)\cdot\big(u_\e(x+t\vec k)-u_\e(x)\big)dsdx=\\
\lim_{\e\to0+}\int\limits_K\int\limits_0^1 \nabla_a W\Big(
P_t(u_\e,x,s,\vec k),u_\e(x)\Big)\cdot
%\frac{1}{\e^Nt}
\bigg(\int\limits_{\R^N}\bigg\{\eta\Big(\frac{y-x-t\vec
k}{\e}\Big)-\eta\Big(\frac{y-x}{\e}\Big)\bigg\}\frac{u(y)}{t\e^N}dy\bigg)dsdx.
\end{multline}
Next, by \er{fgyufghfghGHKK1zzbvqgfgjh}, the Fundamental
Theorem of Calculus, Fubini theorem and integration
by parts we obtain,
%three compacts $K_0\subset\subset K\subset\subset\bar
%K\subset\subset\Omega$
%such that $\dist(K_0,\Omega\setminus K)>0$ and $\dist(K,\Omega\setminus \bar K)>0$ for every $t>0$ such that%$$0<t\leq
%\frac{1}{2}\min\Big\{dist(K_0,\Omega\setminus K),dist(K,\Omega\setminus \bar K)\Big\}$$ we have:
\begin{multline}\label{fgyufghfghGHKK1zzbvqojjj}
I_t=\\ -\lim_{\e\to0+}\int\limits_K\int\limits_0^1 \nabla_a W\Big(
P_t(u_\e,x,s,\vec k)
,u_\e(x)\Big)\cdot\frac{1}{\e^{N+1}}\bigg(\int\limits_{\R^N}\bigg\{\int\limits_0^1\vec
k\cdot\nabla\eta
%\frac{d\eta}{d\vec k}
\Big(\frac{y-x-\tau t\vec
k}{\e}\Big)d\tau\bigg\}u(y)dy\bigg)dsdx\\=
\lim_{\e\to0+}\int_0^1\int_K\int_0^1 \nabla_a W\Big(
P_t(u_\e,x,s,\vec
k),u_\e(x)\Big)\cdot\frac{1}{\e^{N}}\bigg(\int_{\R^N}\eta\Big(\frac{y-x-\tau
t\vec k}{\e}\Big)d\Big[
%\frac{du}{d\vec k}
Du(y)\cdot\vec k\Big]\bigg)dsdxd\tau.
\end{multline}
By \er{fgyufghfghGHKK1zzbvqojjj}, using Fubini Theorem, we
deduce for small $t>0$,
\begin{multline}\label{fgyufghfghGHKK1zzbvq}
I_t=\\
\lim_{\e\to0+}\int_{\R^N}\frac{1}{\e^{N}}\Bigg\{\int_0^1\int_K\int_0^1
\eta\Big(\frac{y-x-\tau t\vec k}{\e}\Big)\nabla_a W\Big(
P_t(u_\e,x,s,\vec k),u_\e(x)\Big)dsdxd\tau\Bigg\}\cdot d\Big[
%\frac{du}{d\vec k}
Du (y)\cdot\vec k\Big]\\=
\lim_{\e\to0+}\int_{\R^N}\frac{1}{\e^{N}}\Bigg\{\int_0^1\int_K\int_0^1
\eta\Big(\frac{x-y-\tau t\vec k}{\e}\Big)\nabla_a W\Big(
P_t(u_\e,y,s,\vec k),u_\e(y)\Big)dsdyd\tau\Bigg\}\cdot d\Big[ Du
%\frac{du}{d\vec k}
(x)\cdot\vec k\Big].
\end{multline}
Performing a change of variables on the r.h.s. of
\er{fgyufghfghGHKK1zzbvq}, using Fubini theorem and denoting for
short
\begin{equation}\label{gjhgjhfghfh}
y=y(\e,x,z,t,\tau,\vec k)=x-\e z-\tau t\vec k,
\end{equation}
we infer
\begin{multline}\label{fgyufghfghGHKKzzbvqhkkhhkjk}
I_t=\lim\limits_{\e\to0+}\int\limits_0^1\int\limits_{\R^N}\eta(z)
\Bigg(\int\limits_{K+\tau t\vec k+\e z}\Bigg\{
\int\limits_0^1\nabla_a W\Big(
%su_\e\big(y+t\vec k\big)+(1-s)u_\e(y)
P_t(u_\e,y,s,\vec k) ,u_\e(y)\Big)ds\Bigg\}\cdot
d\Big[
%\frac{du}{d\vec k}
Du (x)\cdot\vec
k\Big]\Bigg)dzd\tau\\=O\bigg(\big\|Du\big\|\Big(\cup_{\tau\in[0,1]}\big(\partial
K+\tau
t\vec k\big)\Big)\bigg)+\\
\lim\limits_{\e\to0+}\int\limits_0^1\int\limits_{\R^N}\eta(z)
\Bigg(\int\limits_{K+\tau t\vec k}\Bigg\{ \int\limits_0^1\nabla_a
W\Big(
%su_\e\big(y+t\vec k\big)+(1-s)u_\e(y)
P_t(u_\e,y,s,\vec k) ,u_\e(y)\Big)ds\Bigg\}\cdot d\Big[ Du
%\frac{du}{d\vec k}
(x)\cdot\vec k\Big]\Bigg)dzd\tau
\\=O\bigg(\big\|Du\big\|\Big(\cup_{\tau\in[0,1]}\big(\partial
K+\tau
t\vec k\big)\Big)\bigg)+\\
\lim\limits_{\e\to0+}\int\limits_0^1 \Bigg(\int\limits_{K+\tau t\vec
k}\Bigg\{ \int\limits_0^1\int\limits_{\R^N}\eta(z)\nabla_a W\Big(
%su_\e\big(y+t\vec k\big)+(1-s)u_\e(y)
P_t(u_\e,y,s,\vec k) ,u_\e(y)\Big)dzds\Bigg\}\cdot d\Big[ Du
%\frac{du}{d\vec k}
(x)\cdot\vec k \Big]\Bigg)d\tau.
\end{multline}
Using the easy to check fact that
\begin{multline}
\label{ghuutuytjlj} \int_{D_{\tau}\cap J_u}d\Big| Du
%\frac{du}{d\vec k}
(x)\cdot\vec k\Big|=\int_{D_{\tau}\cap
J_u}\Big|u^+(x)-u^-(x)\Big|\,\big|\vec k\cdot\vec
\nu(x)\big|\,d\mathcal{H}^{N-1}(x) =0\quad\text{for
a.e.}\;\,\tau\in(0,1),\\
\quad\text{where}\quad D_{\tau}:=
%\big(\partial K+\tau t\vec k\big)\cup
\big(J_u+\tau t\vec k\big)\cup\big(J_u-(1-\tau) t\vec k\big),
\end{multline}
we decompose \er{fgyufghfghGHKKzzbvqhkkhhkjk} as:
\begin{multline}\label{fgyufghfghGHKKzzbvqjjkjkjljljkjkjh}
I_t=O\bigg(\big\|Du\big\|\Big(\cup_{\tau\in[0,1]}\big(\partial
K+\tau
t\vec k\big)\Big)\bigg)+\\
\lim\limits_{\e\to0+}\int\limits_0^1 \Bigg(\int\limits_{(K+\tau
t\vec k)\cap J_u}\Bigg\{
\int\limits_0^1\int\limits_{\R^N}\eta(z)\nabla_a W\Big(
%su_\e\big(y+t\vec k\big)+(1-s)u_\e(y)
P_t(u_\e,y,s,\vec k) ,u_\e(y)\Big)dzds\Bigg\}\cdot d\Big[ Du
%\frac{du}{d\vec k}
(x)\cdot\vec k \Big]\Bigg)d\tau\\
+ \lim\limits_{\e\to0+}\int\limits_0^1 \Bigg(\int\limits_{(K+\tau
t\vec k)\setminus J_u}\Bigg\{
\int\limits_0^1\int\limits_{\R^N}\eta(z)\nabla_a W\Big(
%su_\e\big(y+t\vec k\big)+(1-s)u_\e(y)
P_t(u_\e,y,s,\vec k) ,u_\e(y)\Big)dzds\Bigg\}\cdot d\Big[ Du
%\frac{du}{d\vec k}
(x) \cdot\vec k\Big]\Bigg)d\tau\\
=O\bigg(\big\|Du\big\|\Big(\cup_{\tau\in[0,1]}\big(\partial K+\tau
t\vec k\big)\Big)\bigg)+\\
\lim\limits_{\e\to0+}\int\limits_0^1 \Bigg(\int\limits_{((K+\tau
t\vec k)\cap J_u)\setminus D_\tau}\Bigg\{
\int\limits_0^1\int\limits_{\R^N}\eta(z)\nabla_a W\Big(
%su_\e\big(y+t\vec k\big)+(1-s)u_\e(y)
P_t(u_\e,y,s,\vec k) ,u_\e(y)\Big)dzds\Bigg\}\cdot d\Big[ Du
%\frac{du}{d\vec k}
(x)\cdot\vec k \Big]\Bigg)d\tau\\
+ \lim\limits_{\e\to0+}\int\limits_0^1 \Bigg(\int\limits_{(K+\tau
t\vec k)\setminus (J_u\cup D_\tau)}\Bigg\{
\int\limits_0^1\int\limits_{\R^N}\eta(z)\nabla_a W\Big(
%su_\e\big(y+t\vec k\big)+(1-s)u_\e(y)
P_t(u_\e,y,s,\vec k) ,u_\e(y)\Big)dzds\Bigg\}\cdot d\Big[ Du
%\frac{du}{d\vec k}
(x)\cdot\vec k \Big]\Bigg)d\tau\\
+ \lim\limits_{\e\to0+}\int\limits_0^1 \Bigg(\int\limits_{((K+\tau
t\vec k)\cap D_\tau)\setminus J_u}\Bigg\{
\int\limits_0^1\int\limits_{\R^N}\eta(z)\nabla_a W\Big(
%su_\e\big(y+t\vec k\big)+(1-s)u_\e(y)
P_t(u_\e,y,s,\vec k) ,u_\e(y)\Big)dzds\Bigg\}\cdot d\Big[ Du
%\frac{du}{d\vec k}
(x)\cdot\vec k \Big]\Bigg)d\tau.
\end{multline}
On the other hand, by \er{jmvnvnbccbvhjhjhhjjkhgjgGHKKzzbvq} we
obtain that for $\mathcal{H}^{N-1}{-a.e.}\;\;x\in\Omega\setminus
J_u$, for every $z\in\R^N$ and for every small $\e>0$ we have
\begin{multline}\label{jmvnvnbccbvhjhjhhjjkhgjgGHKKzzbvqjljj}
\Big|u_\e(x-\e z)-\tilde
u(x)\Big|=\bigg|\int_{\R^N}\eta(y)\Big(u\big(x-\e z+\e y\big)-\tilde
u(x)\Big)dy\bigg|=\int_{\R^N}\eta(y+z)\Big(u\big(x+\e y\big)-\tilde
u(x)\Big)dy\bigg|\\=\bigg|\int_{B_{(1+|z|)}(0)}\eta(y+z)\Big(u\big(x+\e
y\big)-\tilde u(x)\Big)dy\bigg|\leq
\Big(\sup_{y\in\R^N}\big|\eta(y)\big|\Big)\bigg(\int_{B_{(1+|z|)}(0)}\Big|u\big(x+\e
y\big)-\tilde u(x)\Big|dy\bigg).
\end{multline}
Then, by the definition of the approximate limit, for every
$z\in\R^N$ we deduce
\begin{equation}\label{jmvnvnbccbvhjhjhhjjkhgjgGHKKhugjzzbvqlkjkllj}
\lim_{\e\to0^+}u_\e(x-\e z)= \tilde u(x)=\bar u(x)\quad
\quad\text{for}\;\; \mathcal{H}^{N-1}{-a.e.}\;\;x\in\Omega\setminus
J_u\,
\end{equation}
(where $\bar u(x)$ was defined in
\er{jmvnvnbccbvhjhjhhjjkhgjgGHKKhugjzzbvq}). In particular, by
\er{jmvnvnbccbvhjhjhhjjkhgjgGHKKhugjzzbvqlkjkllj} for every small
$t>0$ and for every $\tau\in(0,1)$ we have:
\begin{equation}\label{jmvnvnbccbvhjhjhhjjkhgjgGHKKhugjzzbvqlkjklljjkhkhkkk}
\lim_{\e\to0^+}u_\e\big(x-\e z-\tau t\vec k\big)=\bar u\big(x-\tau
t\vec k\big)\quad \quad\text{for}\;\;
\mathcal{H}^{N-1}{-a.e.}\;\;x\in\big(K+\tau t\vec k\big)\setminus
\big(J_u+\tau t\vec k\big)\end{equation} and
\begin{multline}\label{jjhhkhkhjojojojkkk}
\lim_{\e\to0^+}u_\e\Big(\big(x-\e z-\tau t\vec k\big)+t\vec
k\Big)=\bar u\Big(x+(1-\tau) t\vec k\Big)\\ \text{for}\;\;
\mathcal{H}^{N-1}{-a.e.}\;\;x\in\big(K+\tau t\vec k\big)\setminus
\Big(J_u-(1-\tau) t\vec k\Big).
\end{multline}
Then, using
\er{jmvnvnbccbvhjhjhhjjkhgjgGHKKhugjzzbvqlkjklljjkhkhkkk},
\er{jjhhkhkhjojojojkkk},
\er{jmvnvnbccbvhjhjhhjjkhgjgGHKKhugjzzbvqihhjhjkhkll}, Dominated
Convergence and \er{ghuutuytjlj}, yields
\begin{multline}\label{fgyufghfghGHKKzzbvqjjkjknghfmnmn}
\lim\limits_{\e\to0+}\int\limits_0^1 \Bigg(\int\limits_{((K+\tau
t\vec k)\cap J_u)\setminus D_\tau}\Bigg\{
\int\limits_0^1\int\limits_{\R^N}\eta(z)\nabla_a W\Big(
%su_\e\big(y+t\vec k\big)+(1-s)u_\e(y)
P_t(u_\e,y,s,\vec k) ,u_\e(y)\Big)dzds\Bigg\}\cdot d\Big[ Du
%\frac{du}{d\vec k}
(x)\cdot\vec k \Big]\Bigg)d\tau=\\
\int\limits_0^1\int\limits_{((K+\tau t\vec k)\cap J_u)\setminus
D_\tau}\Bigg\{\int\limits_0^1\nabla_a W\Big(s\bar
u\big(x+(1-\tau)t\vec k\big)+(1-s)\bar u(x-\tau t\vec k),\bar
u(x-\tau t\vec k)\Big)ds\Bigg\}\cdot d\Big[Du
%\frac{du}{d\vec k}
(x)\cdot\vec k\Big]d\tau
\\=
\int\limits_0^1\int\limits_{(K+\tau t\vec k)\cap
J_u}\Bigg\{\int\limits_0^1\nabla_a W\Big(s\bar u\big(x+(1-\tau)t\vec
k\big)+(1-s)\bar u(x-\tau t\vec k),\bar u(x-\tau t\vec
k)\Big)ds\Bigg\}\cdot d\Big[Du
%\frac{du}{d\vec k}
(x)\cdot\vec k\Big]d\tau.
\end{multline}
Similarly, using
\er{jmvnvnbccbvhjhjhhjjkhgjgGHKKhugjzzbvqlkjklljjkhkhkkk},
\er{jjhhkhkhjojojojkkk},
\er{jmvnvnbccbvhjhjhhjjkhgjgGHKKhugjzzbvqihhjhjkhkll} and Dominated
Convergence yields
\begin{multline}\label{fgyufghfghGHKKzzbvqhjjgjgjg}
\lim\limits_{\e\to0+}\int\limits_0^1 \Bigg(\int\limits_{(K+\tau
t\vec k)\setminus (J_u\cup D_\tau)}\Bigg\{
\int\limits_0^1\int\limits_{\R^N}\eta(z)\nabla_a W\Big(
%su_\e\big(y+t\vec k\big)+(1-s)u_\e(y)
P_t(u_\e,y,s,\vec k) ,u_\e(y)\Big)dzds\Bigg\}\cdot d\Big[ Du
%\frac{du}{d\vec k}
(x)\cdot\vec k \Big]\Bigg)d\tau=\\
\int\limits_0^1\int\limits_{(K+\tau t\vec k)\setminus (J_u\cup
D_\tau)}\Bigg\{\int\limits_0^1\nabla_a W\Big(s\bar
u\big(x+(1-\tau)t\vec k\big)+(1-s)\bar u(x-\tau t\vec k),\bar
u(x-\tau t\vec k)\Big)ds\Bigg\}\cdot d\Big[Du
%\frac{du}{d\vec k}
(x)\cdot\vec k\Big]d\tau.
\end{multline}
On the other hand, since the set $(K+\tau t\vec k)\cap D_\tau$ is
$\mathcal{H}^{N-1}$ $\sigma$-finite, by Theorem \ref{vtTh3}
%that is a standard result about $BV$-functions,
we have
\begin{equation}
\label{ghuutuytjljkpjkjjjj} \int_{((K+\tau t\vec k)\cap
D_\tau)\setminus J_u}d\big| Du
%\frac{du}{d\vec k}
(x)\cdot\vec k\big|=0,
\end{equation}
and in particular,
\begin{multline}\label{fguyyytyutytu}
\lim\limits_{\e\to0+}\int\limits_0^1 \Bigg(\int\limits_{((K+\tau
t\vec k)\cap D_\tau)\setminus J_u}\Bigg\{
\int\limits_0^1\int\limits_{\R^N}\eta(z)\nabla_a W\Big(
%su_\e\big(y+t\vec k\big)+(1-s)u_\e(y)
P_t(u_\e,y,s,\vec k) ,u_\e(y)\Big)dzds\Bigg\}\cdot d\Big[ Du
%\frac{du}{d\vec k}
(x)\cdot\vec k \Big]\Bigg)d\tau=0.
\end{multline}
Thus, inserting \er{fgyufghfghGHKKzzbvqjjkjknghfmnmn},
\er{fgyufghfghGHKKzzbvqhjjgjgjg} and \er{fguyyytyutytu} into
\er{fgyufghfghGHKKzzbvqjjkjkjljljkjkjh} yields
\begin{multline}\label{fgyufghfghGHKKzzbvqiojoj}
I_t
%\lim\limits_{\e\to0+}\int\limits_0^1\int\limits_{\R^N}\eta(z)\Bigg(\int\limits_{K+\tau t\vec k+\e z}\Bigg\{\int\limits_0^1\nabla_a W\Big(
%su_\e\big(y+t\vec k\big)+(1-s)u_\e(y)P_t(u_\e,y,s,\vec k) ,u_\e(y)\Big)ds\Bigg\}\cdot d\Big[\frac{du}{d\vec k}(x)\Big]\Bigg)dzd\tau\\
=O\bigg(\big\|Du\big\|\Big(\cup_{\tau\in[0,1]}\big(\partial K+\tau
t\vec k\big)\Big)\bigg)+\\
\int\limits_0^1\int\limits_{(K+\tau t\vec k)\cap
J_u}\Bigg\{\int\limits_0^1\nabla_a W\Big(s\bar u\big(x+(1-\tau)t\vec
k\big)+(1-s)\bar u(x-\tau t\vec k),\bar u(x-\tau t\vec
k)\Big)ds\Bigg\}\cdot d\Big[Du(x)\cdot\vec k\Big]d\tau+\\
\int\limits_0^1\int\limits_{(K+\tau t\vec k)\setminus (J_u\cup
D_\tau)}\Bigg\{\int\limits_0^1\nabla_a W\Big(s\bar
u\big(x+(1-\tau)t\vec k\big)+(1-s)\bar u(x-\tau t\vec k),\bar
u(x-\tau t\vec k)\Big)ds\Bigg\}\cdot d\Big[Du
%\frac{du}{d\vec k}
(x)\cdot\vec k\Big]d\tau.
\end{multline}
Then, using again \er{ghuutuytjljkpjkjjjj} in
\er{fgyufghfghGHKKzzbvqiojoj} we get
\begin{multline}\label{fgyufghfghGHKKzzbvq}
I_t
%\lim\limits_{\e\to0+}\int\limits_0^1\int\limits_{\R^N}\eta(z)\Bigg(\int\limits_{K+\tau t\vec k+\e z}\Bigg\{\int\limits_0^1\nabla_a W\Big(
%su_\e\big(y+t\vec k\big)+(1-s)u_\e(y)P_t(u_\e,y,s,\vec k) ,u_\e(y)\Big)ds\Bigg\}\cdot d\Big[\frac{du}{d\vec k}(x)\Big]\Bigg)dzd\tau\\
=O\bigg(\big\|Du\big\|\Big(\cup_{\tau\in[0,1]}\big(\partial K+\tau
t\vec k\big)\Big)\bigg)+\\
\int\limits_0^1\int\limits_{K+\tau t\vec
k}\Bigg\{\int\limits_0^1\nabla_a W\Big(s\bar u\big(x+(1-\tau)t\vec
k\big)+(1-s)\bar u(x-\tau t\vec k),\bar u(x-\tau t\vec
k)\Big)ds\Bigg\}\cdot d\Big[Du
%\frac{du}{d\vec k}
(x)\cdot\vec k\Big]d\tau.
\end{multline}
On the other hand, by Theorem 3.108 and Remark 3.109 from \cite{amb}
we deduce that
%there exists $R_0<R$ such that
%%%%%%%%%%%%%%%%%%%%%%%%%%%%%%%%%%%%%%%%%%%%%%%%%%%%%%%%%%%%%%%%%%%%%%%%%%%%%%%%%%%%%%%%%%%%%%%%%%%%%%%%%%%%%%%%%%%%%%%%%%%%%%%%%%%%%%%%%%%%%%%%%%%%%%%%%%%%%%%%%%%%%%%%%%%%%%%%%%%%%%%%%%%%%%%%%%%%%%%%%%%%%%%%%%%%%%%%%%%%%%%%%%%%%%%%%%%%%%%%%%%%%%%%%%%%%%%%%%%%%%%%%%%%%%%%%%%%%%%%%%%%%%%%%%%%%%%%%%%%%%%%%%%%
\begin{equation}
\label{shrppppp8kkkkGG} \left\{
\begin{aligned}
\lim\limits_{\rho\to0^+}
%\frac{1}{\rho}
\int_{0}^{1}
%\rho
\bigg(\big|\bar u(x+\rho s\vec k)-\tilde u(x)\big|&+\big|\bar
u(x-\rho
s\vec k)-\tilde u(x)\big|\bigg)\,ds=0,\\
%\text{for }
&\mathcal{H}^{N-1}\text{ a.e. in }\Omega\setminus  J_u\\
\lim\limits_{\rho\to0^+}
%\frac{1}{\rho}
\int_{0}^{1}
%\rho
\bigg(\big|\bar u(x+\rho s\vec k)-u^+(x)\big|&+\big|\bar u(x-\rho
s\vec k)-u^-(x)\big|\bigg)\,ds=0,\\
%\text{for }
&\mathcal{H}^{N-1}\text{ a.e. in }\Big\{x\in J_u\,:\;\vec
k\cdot\vec\nu(x)>0\Big\}\\ \lim\limits_{\rho\to 0^+}
%\frac{1}{\rho}
\int_{0}^{1}
%\rho
\bigg(\big|\bar u(x+\rho s\vec k)-u^-(x)\big|&+\big|\bar u(x-\rho
s\vec k)-u^+(x)\big|\bigg)\,ds=0,\\
%\text{for }
&\mathcal{H}^{N-1}\text{ a.e. in }\Big\{x\in J_u\,:\;\vec
k\cdot\vec\nu(x)<0\Big\}.
%\\ \lim\limits_{\rho\to
%0^+}\frac{1}{\rho}\int_{-\rho}^{0}\big|f\big(g(x')+s,x'\big)-f^-\big(g(x'),x'\big)\big|\,ds=0\quad\quad\text{for
%} \mathcal{L}^{N-1}\text{ a.e. }x'\in\mathcal{U}\,.
\end{aligned}
\right.
\end{equation}
Thus, since $\|Du\|(\partial K)=0$, by \er{fgyufghfghGHKKzzbvq}, the
first equation in \er{shrppppp8kkkkGG}, Dominated
Convergence and the properties $W(a,b)=W(b,a)$, $W(a,a)=0$
and $\nabla W(a,a)=0$ (since $W\geq 0$), we obtain
\begin{multline}\label{fgyufghfghjgghgjkhkkhhkggkjgmhkjjuyhghgGHKKzzbvqujuouj}
%I=
\lim_{t\to0^+}I_t=\lim_{t\to0^+}\frac{1}{t}\int_KW\Big(u(x+t\vec
k),u(x)\Big)dx=\\
\lim_{t\to0^+}\int\limits_0^1\int\limits_{J_u\cap
K}\Bigg\{\int\limits_0^1\nabla_a W\Big(s\bar u\big(x+(1-\tau)t\vec
k\big)+(1-s)\bar u(x-\tau t\vec k),\bar u(x-\tau t\vec
k)\Big)ds\Bigg\}\cdot d\Big[Du
%\frac{du}{d\vec k}
(x)\cdot\vec k\Big]d\tau.
\end{multline}
Then by inserting the second two equations in \er{shrppppp8kkkkGG}
into \er{fgyufghfghjgghgjkhkkhhkggkjgmhkjjuyhghgGHKKzzbvqujuouj} and
using Dominated Convergence and
Theorem \ref{vtTh3} we deduce:
\begin{multline}\label{fgyufghfghjgghgjkhkkhhkggkjgmhkjjuyhghgGHKKzzbvqljjjj}
%I=
\lim_{t\to0^+}I_t=\\
%=\lim_{t\to0^+}\frac{1}{t}\int_KW\Big(u(x+t\vec k),u(x)\Big)dx=\\
\int\limits_{J_u\cap K}\int\limits_0^1\nabla_a
W\Big(su^+(x)+(1-s)u^-(x),u^-(x)\Big)\cdot\big(u^+(x)-u^-(x)\big)\Big(\max\big\{\vec
k\cdot\vec\nu(x),0\big\}\Big)dsd\mathcal{H}^{N-1}(x)+\\
\int\limits_{J_u\cap K}\int\limits_0^1\nabla_a
W\Big(su^-(x)+(1-s)u^+(x),u^+(x)\Big)\cdot\big(u^-(x)-u^+(x)\big)\Big(\max\big\{-\vec
k\cdot\vec\nu(x),0\big\}\Big)dsd\mathcal{H}^{N-1}(x).
\end{multline}
Then, using the Fundamental Theorem of Calculus in
\er{fgyufghfghjgghgjkhkkhhkggkjgmhkjjuyhghgGHKKzzbvqljjjj} gives
\begin{multline}\label{fgyufghfghjgghgjkhkkhhkggkjgmhkjjuyhghgGHKKzzbvq}
%I=
\lim_{t\to0^+}I_t= \int_{J_u\cap
K}W\Big(u^+(x),u^-(x)\Big)\Big(\max\big\{\vec
k\cdot\vec\nu(x),0\big\}\Big)d\mathcal{H}^{N-1}(x)\\+\int_{J_u\cap
K}W\Big(u^-(x),u^+(x)\Big)\Big(\max\big\{-\vec
k\cdot\vec\nu(x),0\big\}\Big)d\mathcal{H}^{N-1}(x)\\=\int_{J_u\cap
K}W\Big(u^+(x),u^-(x)\Big)\big|\vec
k\cdot\vec\nu(x)\big|d\mathcal{H}^{N-1}(x).
\end{multline}
The desired estimate \er{fgyufghfghjgghgjkhkkGHKKzzbvq} follows
immediately from
\er{fgyufghfghjgghgjkhkkhhkggkjgmhkjjuyhghgGHKKzzbvq} and
\er{fgyufghfghjgghgjkhkkGHGHKKzzbvq} is deduced from the particular
case $W(a,b)=|a-b|^q$. Moreover, for any compact set
$K\subset\subset\Omega$, we can choose
$\Omega_1\subset\subset\Omega$ such that $K\subset\subset\Omega_1$
and then for every small $t>0$ we clearly have
\begin{multline}\label{hgjhkjgfgjffzzbvq}
0\leq\frac{1}{t}\int_K\Big|u(x+t\vec
k)-u(x)\Big|^qdx\leq2^{q-1}\|u\|^{q-1}_{L^\infty(K)}\int_K\frac{1}{t}\Big|u(x+t\vec
k)-u(x)\Big|dx\\ \leq
2^{q-1}\|u\|^{q-1}_{L^\infty(K)}\|u\|_{BV(\Omega_1)}.
\end{multline}
Thus by dominated convergence we get
\begin{multline}\label{fgyufghfghjgghgjkhkkGHGHKKjjjjkjkkjjhjkzzbvq}
A_{u,q}\big(K\big)=\lim\limits_{t\to
0^+}\Bigg(\int_{{B_1(0)}}\int_K\frac{1}{t|z|}\Big|u(x+t
z)-u(x)\Big|^qdxdz\Bigg)\\=\Bigg(\int_{{B_1(0)}}\bigg(\int_{J_u\cap
K}\Big|u^+(x)-u^-(x)\Big|^q\bigg|\frac{z}{|z|}\cdot\vec\nu(x)\bigg|d\mathcal{H}^{N-1}(x)\bigg)dz\Bigg)=\\
\bigg(\int_{{B_1(0)}}\frac{|z_1|}{|z|}dz\bigg)\bigg(\int_{J_u\cap
K}\Big|u^+(x)-u^-(x)\Big|^qd\mathcal{H}^{N-1}(x)\bigg)\\=
\bigg(\frac{1}{N}\int_{S^{N-1}}|z_1|d\mathcal{H}^{N-1}(z)\bigg)\bigg(\int_{J_u\cap
K}\Big|u^+(x)-u^-(x)\Big|^qd\mathcal{H}^{N-1}(x)\bigg),
\end{multline}
and
\er{fgyufghfghjgghgjkhkkGHGHKKjjjjkjkkjzzbvq} follows.
\end{proof}
%
%
%
\begin{comment}
\begin{theorem}\label{ghgghgghjjkjkzzbvqKK}
Let $\Omega\subset\R^N$ be an open set and let $u\in
BV_{loc}(\Omega,\R^d)\cap L^\infty_{loc}(\Omega,\R^d)$. Then, for
every $q>1$ we have $u\in {BV}^q_{loc}(\Omega,\R^d)$ and for every
compact set $K\subset\subset\Omega$ with Lipschitz boundary such
that $\|Du\|(\partial K)=0$ we have
\begin{equation}\label{fgyufghfghjgghgjkhkkGHGHKKjjjjkjkkjkmmlmjijilzzbvqKK}
A_{u,q}\big(K\big)=C_N\int_{J_u\cap
K}\Big|u^+(x)-u^-(x)\Big|^qd\mathcal{H}^{N-1}(x),
\end{equation}
with  $C_N$ as in \er{fgyufghfghjgghgjkhkkGHGHKKggkhhjoozzbvq}.        Moreover, if in
addition $u\in BV(\Omega,\R^d)\cap L^\infty(\Omega,\R^d)$, then
for every $q>1$ we have
%$u(x)\in ({BV}^q)'(\ov\Omega,\R^d)$ and
\begin{equation}\label{fgyufghfghjgghgjkhkkGHGHKKjjjjkjkkjkmmlmjijiluuizzbvqKK}
A_{u,q}\big(\Omega\big)=C_N\int_{J_u}\Big|u^+(x)-u^-(x)\Big|^qd\mathcal{H}^{N-1}(x).
\end{equation}
Finally, if $\Omega$ is an open set with bounded Lipschitz boundary
and $u\in BV(\Omega,\R^d)\cap L^\infty(\Omega,\R^d)$ then we have
$u\in {BV}^q(\Omega,\R^d)$ for every $q>1$ and
\begin{equation}\label{fgyufghfghjgghgjkhkkGHGHKKjjjjkjkkjkmmlmjijiluuizziihhhjbvqKK}
\hat
A_{u,q}\big(\Omega\big)=C_N\int_{J_u}\Big|u^+(x)-u^-(x)\Big|^qd\mathcal{H}^{N-1}(x)=A_{u,q}\big(\Omega\big).
\end{equation}
\end{theorem}
\end{comment}
%
%
%
\begin{proof}[Proof of Theorem \ref{ghgghgghjjkjkzzbvq}]
%[Proof of Theorem \ref{ghgghgghjjkjkzzbvq}]
Identities \er{fgyufghfghjgghgjkhkkGHGHKKjjjjkjkkjkmmlmjijilzzbvqKK}
and \er{fgyufghfghjgghgjkhkkGHGHKKjjjjkjkkjkmmlmjijiluuizzbvqKK}
follow from Proposition \ref{hgugghghhffhfhKKzzbvq}. For every $\vec k\in S^{N-1}$,  every open
$\Omega_1\subset\Omega$ such that $u\in BV(\Omega_1,\R^d)\cap
L^\infty(\Omega_1,\R^d)$, every $K\subset\subset\Omega_1$ and
$0<t<\dist(K,\R^N\setminus\Omega_1)$ we have
\begin{multline}\label{hgjhkjgfgjffkjlljjlkllkjjzzbvq}
0\leq\frac{1}{t}\int_K\Big|u(x+t\vec
k)-u(x)\Big|^qdx\leq2^{q-1}\|u\|^{q-1}_{L^\infty(K)}\int_K\frac{1}{t}\Big|u(x+t\vec
k)-u(x)\Big|dx\\ \leq
2^{q-1}\|u\|^{q-1}_{L^\infty(K)}\|u\|_{BV(\Omega_1)}.
\end{multline}
Therefore, we obtain $BV_{loc}(\Omega,\R^d)\cap
L^\infty_{loc}(\Omega,\R^d)\subset {BV}^q_{loc}(\Omega,\R^d)$.

 Finally, if $\Omega$ is
an open set with bounded Lipschitz boundary and  $u\in
BV(\Omega,\R^d)\cap L^\infty(\Omega,\R^d)$, then
%by Theorem \ref{ghgghgghzzZZzzbvq} we obtain
%BV(\Omega,\R^d)\cap L^\infty(\Omega,\R^d)\subset
%$u(x)\in{BV}^q(\ov\Omega,\R^d)$. Moreover, in this case
we can extend the function $u(x)$ to all of $\R^N$ in such a way
that $u\in BV(\R^N,\R^d)\cap L^\infty(\R^N,\R^d)$ and $\|D
u\|(\partial\Omega)=0$. Next in the case of bounded $\O$
%if $G\subset\R^N$ is any open set such that $\ov\Omega\subset G$ and $dist(\Omega,\R^N\setminus G)>0$, then
clearly, we have
\begin{equation}\label{hugyfgoubvq}
A_{u,q}(\Omega)\leq\hat A_{u,q}(\Omega)\leq A_{u,q}(\ov\Omega).
\end{equation}
Thus, since $\|D u\|(\partial\Omega)=0$, combining \er{hugyfgoubvq}
together with
\er{fgyufghfghjgghgjkhkkGHGHKKjjjjkjkkjkmmlmjijilzzbvqKK} and
\er{fgyufghfghjgghgjkhkkGHGHKKjjjjkjkkjkmmlmjijiluuizzbvqKK} yields
$A_{u,q}(\Omega)=\hat A_{u,q}(\Omega)=A_{u,q}(\ov\Omega)$, and in
particular, $u\in{BV}^q(\Omega,\R^d)$. On the hand, if $\O$ is
unbounded
%then there exists
consider a strictly increasing positive sequence
$R_n\uparrow\infty$, such that $\|D
u\|\big(\partial\Omega\cup\partial B_{R_n}(0)\big)=0$. Then,
similarly to \er{hgjhkjgfgjffkjlljjlkllkjjzzbvq}
%for every compact $K\subset\subset\O$
we have
\begin{multline}\label{hugyfgoubvqjjjjj}
%A_{u,q}(K)\leq
A_{u,q}(\Omega)\leq\hat A_{u,q}(\Omega)\leq
A_{u,q}\big(\ov\Omega\cap \ov B_{R_{n+2}}(0)\big)+
A_{u,q}\big(\R^N\setminus \ov B_{R_{n+1}}(0)\big)\\ \leq
A_{u,q}\big(\ov\Omega\cap \ov B_{R_{n+2}}(0)\big)+
\frac{1}{\mathcal{L}^N({B_1(0)})}B_{u,q}\big(\R^N\setminus \ov
B_{R_{n+1}}(0)\big)\\ \leq A_{u,q}\big(\ov\Omega\cap \ov
B_{R_{n+2}}(0)\big)+
\frac{2^{q-1}}{\mathcal{L}^N({B_1(0)})}\|u\|^{q-1}_{L^\infty(\R^N)}\|u\|_{BV(\R^N\setminus
\ov B_{R_{n}}(0))} .
\end{multline}
Thus letting $n$ tend to $\infty$ in \er{hugyfgoubvqjjjjj} and using
\er{fgyufghfghjgghgjkhkkGHGHKKjjjjkjkkjkmmlmjijilzzbvqKK} and
\er{fgyufghfghjgghgjkhkkGHGHKKjjjjkjkkjkmmlmjijiluuizzbvqKK} again
yields $$A_{u,q}(\Omega)=\hat
A_{u,q}(\Omega)=\lim_{n\to\infty}A_{u,q}\big(\ov\Omega\cap \ov
B_{R_{n+2}}(0)\big),$$ that completes the proof.
%and in particular, $u\in{BV}^q(\Omega,\R^d)$.
\end{proof}

The next Lemma contains the main ingredient of the proof of Theorem~\ref{huyhuyuughhjhhjjhhjjjhhjhjjhhhjzzbvq}.
\begin{lemma}\label{huyhuyuughhjhhjjhhjjjzzbvq}
For any open set $\Omega\subset\R^N$, $q>1$ and $u\in W^{\frac{1}{q},q}(\Omega,\R^d)$ we have $u\in
BV^q(\Omega,\R^d)$. Moreover,
\begin{equation}\label{fgyufghfghggjhjkkzzbvqkjkjjhg}
\bar A_{u,q}\big(\Omega\big)\leq
\int_\Omega\int_\Omega\frac{\big|u(x)-u(y)\big|^q}{|x-y|^{N+1}}dydx
\end{equation}
and
%for every compact $K\subset\subset\Omega$ we have
\begin{equation}\label{fgyufghfghggjhjkkzzbvq}
\hat A_{u,q}\big(\Omega\big)=0.
\end{equation}
%and
%\begin{equation}\label{fgyufghfghjgghgjkhkkGHGHKKokuhhhllzz}
%\tilde A_{u,q,\Omega}\big(K\big)\leq\int_\Omega\int_\Omega\frac{\big|u(x)-u(y)\big|^q}{|x-y|^{N+1}}dxdy<+\infty.
%\end{equation}
\end{lemma}
\begin{proof}
For every $\e\in(0,1)$ we have
\begin{multline}\label{fgyufghfghjgghgjkhkkGHGHKKokuhhhugugzzkhhbvq}
\infty>T=
\int_\Omega\bigg(\int_\Omega\frac{\big|u(x)-u(y)\big|^q}{|x-y|^{N+1}}dy\bigg)dx\geq\int_\Omega\bigg(\int_{B_\e(x)\cap\Omega}\frac{\big|u(x)-u(y)\big|^q}{|x-y|^{N+1}}dy\bigg)dx\\
\geq \int_\Omega\bigg(\int_{B_\e(x)\cap\Omega}\frac{1}{\e^N}
\,\frac{\big|u(x)-u(y)\big|^q}{|x-y|}dy\bigg)dx.
\end{multline}
In particular, we deduce \er{fgyufghfghggjhjkkzzbvqkjkjjhg}. Next,
%we deduce \er{fgyufghfghjgghgjkhkkGHGHKKokuhhhllzz}. Moreover,
by \er{fgyufghfghjgghgjkhkkGHGHKKokuhhhugugzzkhhbvq} we infer
\begin{multline}\label{fgyufghfghjgghgjkhkkGHGHKKokuhhhugugjhkzzggbvq}
\limsup_{\e\to0^+}\Bigg(\int_\Omega\bigg(\int_{B_\e(x)\cap\Omega}\frac{1}{\e^N}
\,\frac{\big|u(x)-u(y)\big|^q}{|x-y|}dy\bigg)dx\Bigg) \leq\\
\limsup_{\e\to0^+}\Bigg(\int_\Omega\bigg(\int_{B_\e(x)\cap\Omega}\frac{\big|u(x)-u(y)\big|^q}{|x-y|^{N+1}}dy\bigg)dx\Bigg)
\leq
\int_\Omega\bigg(\int_\Omega\frac{\big|u(x)-u(y)\big|^q}{|x-y|^{N+1}}dy\bigg)dx<\infty.
\end{multline}
On the other hand, dominated convergence implies that
$$\limsup_{\e\to0^+}\Bigg(\int_\Omega\bigg(\int_{B_\e(x)\cap\Omega}\frac{\big|u(x)-u(y)\big|^q}{|x-y|^{N+1}}dy\bigg)dx\Bigg)=0.$$
Plugging the above in
\er{fgyufghfghjgghgjkhkkGHGHKKokuhhhugugjhkzzggbvq} yields
\begin{equation}\label{fgyufghfghjgghgjkhkkGHGHKKokuhhhugugjhkhjgjjzzggbvq}
\limsup_{\e\to0^+}\Bigg(\int_\Omega\bigg(\int_{B_\e(x)\cap\Omega}\frac{1}{\e^N}
\,\frac{\big|u(x)-u(y)\big|^q}{|x-y|}dy\bigg)dx\Bigg)=0,
\end{equation}
and \er{fgyufghfghggjhjkkzzbvq} follows.
\end{proof}

We recall below the definitions of the spaces $B$ and $B_0$ from
\cite{BBM3}.
\begin{definition}\label{gvyfhgfhfffbvq}
For every $x=(x_1,\ldots,x_N)\in\R^N$ and every $\e>0$ consider the
$\e$-cube:
\begin{equation}\label{jghjghjgghghghghbvq}
Q_\e(x):=\Big\{z=(z_1,\ldots,z_N)\in\R^N\,:\;|z_j-x_j|<\frac{\e}{2}\Big\}.
\end{equation}
Then, for any open set $\Omega\subset\R^N$ and any $\e>0$ denote by
$\mathcal{R}_\e(\Omega)$ the set of all collections of disjoint
$\e$-cubes $\big\{Q_\e(x_j)\big\}_{j=1}^{m}$ contained in $\Omega$
with $m\in\Big[0,\frac{1}{\e^{N-1}}\Big]$, such that
$\bigcup_{j=1}^{m} Q_\e(x_j)\subseteq\Omega$ and $Q_\e(x_k)\cap
Q_\e(x_j)=\emptyset$ whenever $k\neq j$. Furthermore, for every
small $\e>0$ and every $u\in L^1_{loc}(\Omega,\R^d)$ define
\begin{align}\label{jghjghjgghghghghjhkhhbvq}
&[u]_{\e}(\Omega):=
\sup\Bigg\{\sum_{j=1}^{m}\bigg(\frac{\e^{N-1}}{\big(\mathcal{L}^N(Q_\e(x_j))\big)^2}
\!\!\iint_{(Q_\e(x_j))^2}\!\Big|u(y)-u(x)\Big|dydx\bigg)\,:\,\{Q_\e(x_j)\}_{j=1}^{m}\in
\mathcal{R}_\e(\Omega)\Bigg\},\\
\label{jghjghjgghghghghjhkhhhihbvquiyuyu}
&|u|_{B(\Omega,\R^d)}:=\sup\limits_{\e\in(0,1)}[u]_{\e}(\Omega),\\
\intertext{and}
\label{jghjghjgghghghghjhkhhhihbvq}
&[u](\Omega):=\limsup\limits_{\e\to0^+}[u]_{\e}(\Omega).
\end{align}
Define the spaces
\begin{equation}\label{jghjghjgghghghghjhkhhhihiiubvq}
\begin{aligned}
B(\Omega,\R^d)&:=\big\{u\in
L^1(\Omega,\R^d)\,:\;|u|_{B(\Omega,\R^d)}<\infty\big\}=\big\{u\in
L^1(\Omega,\R^d)\,:\;[u](\Omega)<\infty\big\},\\
B_0(\Omega,\R^d)&:=\big\{u\in
L^1(\Omega,\R^d)\,:\;[u](\Omega)=0\big\}\,\subset\,B(\Omega,\R^d).
\end{aligned}
\end{equation}
Then,
%$BV^q(\ov\Omega,\R^d)\subseteq BV^q(\Omega,\R^d)$. Moreover,
$B(\Omega,\R^d)$ is a normed linear space with the norm
\begin{equation}\label{hgghggghkghghhgjjjj}
\|u\|_{B(\Omega,\R^d)}:=|u|_{B(\Omega,\R^d)}+\|u\|_{L^1(\Omega,\R^d)},
\end{equation}
and $B_0(\Omega,\R^d)$ is a closed subspace of
$B(\Omega,\R^d)$.
\end{definition}
\begin{lemma}\label{guytuttbvq}
%For every $x=(x_1,\ldots,x_N)\in\R^N$ and every $\e>0$ consider the
%$\e$-cube defined by \er{jghjghjgghghghghbvq}.
%\begin{equation}
%Q_\e(x):=\Big\{z=(z_1,\ldots,z_N)\in\R^N\,:\;|z_j-x_j|<\frac{\e}{2}\Big\}.
%\end{equation}
For any open set $\Omega\subset\R^N$, $q\geq 1$, $u\in
L^q(\Omega,\R^d)$, $\e>0$, an integer
$m\in\Big[0,\frac{1}{\e^{N-1}}\Big]$ and arbitrary $m$ points
$\big\{x_j\big\}_{j=1}^{m}\subset\Omega$, such that
$\bigcup_{j=1}^{m} Q_\e(x_j)\subset\Omega$ and $Q_\e(x_k)\cap
Q_\e(x_j)=\emptyset$ for $k\neq j$, we have
%$u\in BV^q(\ov\Omega,\R^d)$
\begin{multline}\label{fgyufghfghjgghgjkhkkhhkggkjgmhkjjuyhghgGHKKvbvbvvzzjjjkkpkbvq}
\sum_{j=1}^{m}\e^{N-1}\Bigg(\frac{1}{\big(\mathcal{L}^N(Q_\e(x_j))\big)^2}\int_{Q_\e(x_j)}\int_{Q_\e(x_j)}\Big|u(y)-u(x)\Big|dydx\Bigg)\\
\leq
N^{\frac{N+1}{2q}}\Bigg(\int\limits_{\Omega}\int\limits_{B_{\e'}(x)\cap\Omega}\frac{1}{(\e')^{N}}\,\frac{\big|u(y)-u(x)\big|^q}{|y-x|}
dydx\Bigg)^{\frac{1}{q}},
\end{multline}
where $\e':=\e\sqrt{N}$.
\end{lemma}
\begin{proof}
By H\"{o}lder inequality, we have
\begin{multline}\label{fgyufghfghjgghgjkhkkhhkggkjgmhkjjuyhghgGHKKvbvbvvzzjjjkkpkhjhbvq}
\Bigg(\frac{1}{\big(\mathcal{L}^N(Q_\e(x_j))\big)^2}\int_{Q_\e(x_j)}\int_{Q_\e(x_j)}\Big|u(y)-u(x)\Big|dydx\Bigg)
\leq\\
\Bigg(\frac{1}{\big(\mathcal{L}^N(Q_\e(x_j))\big)^2}\int_{Q_\e(x_j)}\int_{Q_\e(x_j)}\Big|u(y)-u(x)\Big|^q
dydx\Bigg)^{\frac{1}{q}}=\\
\Bigg(\frac{1}{\e^{2N}}\int_{Q_\e(x_j)}\int_{Q_\e(x_j)}\Big|u(y)-u(x)\Big|^q
dydx\Bigg)^{\frac{1}{q}}\leq\Bigg(\frac{\sqrt{N}}{\e^{N-1}}\int_{Q_\e(x_j)}\int_{Q_\e(x_j)}\frac{1}{\e^{N}}\frac{\big|u(y)-u(x)\big|^q}{|y-x|}
dydx\Bigg)^{\frac{1}{q}}.
\end{multline}
On the other hand, by the H\"{o}lder's inequality (on finite sums)
we have
%
%
%
\begin{comment}
the function $g(s)=s^{\frac{1}{q}}$ is concave for $s\in[0,\infty)$
and therefore, for all $s_1,\ldots s_m\in[0,\infty)$ we have
\begin{equation*}
%\label{jmvnvnbccbvhjhjhhjjkhgjgGHKKjhhjzzijhjjh}
\frac{1}{m}\bigg(\sum_{j=1}^{m}s^{\frac{1}{q}}_j\bigg)\leq\Bigg(\frac{1}{m}\bigg(\sum_{j=1}^{m}s_j\bigg)\Bigg)^{\frac{1}{q}}\,,
\end{equation*}
whence
\end{comment}
%
%
%
\begin{equation}\label{jmvnvnbccbvhjhjhhjjkhgjgGHKKjhhjzzijhjjhiuybvq}
\bigg(\sum_{j=1}^{m}s^{\frac{1}{q}}_j\bigg)\leq
m^{\frac{q-1}{q}}\bigg(\sum_{j=1}^{m}s_j\bigg)^{\frac{1}{q}}.
\end{equation}
Therefore, by
\er{fgyufghfghjgghgjkhkkhhkggkjgmhkjjuyhghgGHKKvbvbvvzzjjjkkpkhjhbvq}
and \er{jmvnvnbccbvhjhjhhjjkhgjgGHKKjhhjzzijhjjhiuybvq} we have
\begin{multline}\label{fgyufghfghjgghgjkhkkhhkggkjgmhkjjuyhghgGHKKvbvbvvzzjjjkkpkhjhjjhbvq}
\sum_{j=1}^{m}\e^{N-1}\Bigg(\frac{1}{\big(\mathcal{L}^N(Q_\e(x_j))\big)^2}\int_{Q_\e(x_j)}\int_{Q_\e(x_j)}\Big|u(y)-u(x)\Big|dydx\Bigg)
\\ \leq\sum_{j=1}^{m}\e^{N-1}\Bigg(\frac{\sqrt{N}}{\e^{N-1}}\int_{Q_\e(x_j)}\int_{Q_\e(x_j)}\frac{1}{\e^{N}}\frac{\big|u(y)-u(x)\big|^q}{|y-x|}
dydx\Bigg)^{\frac{1}{q}}\\
\leq\e^{N-1}m^{\frac{q-1}{q}}\Bigg(\sum_{j=1}^{m}\frac{\sqrt{N}}{\e^{N-1}}\int_{Q_\e(x_j)}\int_{Q_\e(x_j)}\frac{1}{\e^{N}}\frac{\big|u(y)-u(x)\big|^q}{|y-x|}
dydx\Bigg)^{\frac{1}{q}}\\=
\Big(\e^{N-1}m\Big)^{\frac{q-1}{q}}N^{\frac{1}{2q}}\Bigg(\sum_{j=1}^{m}\int_{Q_\e(x_j)}\int_{Q_\e(x_j)}\frac{1}{\e^{N}}\frac{\big|u(y)-u(x)\big|^q}{|y-x|}
dydx\Bigg)^{\frac{1}{q}}.
\end{multline}
By
\er{fgyufghfghjgghgjkhkkhhkggkjgmhkjjuyhghgGHKKvbvbvvzzjjjkkpkhjhjjhbvq}
and our assumption  $m\leq\frac{1}{\e^{N-1}}$ it follows that
\begin{multline}\label{fgyufghfghjgghgjkhkkhhkggkjgmhkjjuyhghgGHKKvbvbvvzzjjjkkpkhjhjjhhjhkbvq}
\sum_{j=1}^{m}\e^{N-1}\Bigg(\frac{1}{\big(\mathcal{L}^N(Q_\e(x_j))\big)^2}\int_{Q_\e(x_j)}\int_{Q_\e(x_j)}\Big|u(y)-u(x)\Big|dydx\Bigg)
\\ \leq
N^{\frac{1}{2q}}\Bigg(\sum_{j=1}^{m}\int_{Q_\e(x_j)}\int_{Q_\e(x_j)}\frac{1}{\e^{N}}\frac{\big|u(y)-u(x)\big|^q}{|y-x|}
dydx\Bigg)^{\frac{1}{q}}.
\end{multline}
Since for every $x\in Q_\e(x_j)$ we have
$Q_\e(x_j)\subset B_{(\e\sqrt{N})}(x)\cap\Omega$, we get from
\er{fgyufghfghjgghgjkhkkhhkggkjgmhkjjuyhghgGHKKvbvbvvzzjjjkkpkhjhjjhhjhkbvq}
that
\begin{multline}\label{fgyufghfghjgghgjkhkkhhkggkjgmhkjjuyhghgGHKKvbvbvvzzjjjkkpkhjhjjhhjhkuuiuibvq}
\sum_{j=1}^{m}\e^{N-1}\Bigg(\frac{1}{\big(\mathcal{L}^N(Q_\e(x_j))\big)^2}\int_{Q_\e(x_j)}\int_{Q_\e(x_j)}\Big|u(y)-u(x)\Big|dydx\Bigg)
\\ \leq
N^{\frac{1}{2q}}\Bigg(\sum\limits_{j=1}^{m}\int\limits_{Q_\e(x_j)}\bigg(\int\limits_{B_{(\e\sqrt{N})}(x)\cap\Omega}\frac{1}{\e^{N}}\frac{\big|u(y)-u(x)\big|^q}{|y-x|}
dy\bigg)dx\Bigg)^{\frac{1}{q}}=\\
N^{\frac{N+1}{2q}}\Bigg(\sum\limits_{j=1}^{m}\int\limits_{Q_\e(x_j)}\bigg(\int\limits_{B_{(\e\sqrt{N})}(x)\cap\Omega}\frac{1}{(\e\sqrt{N})^{N}}\,\frac{\big|u(y)-u(x)\big|^q}{|y-x|}
dy\bigg)dx\Bigg)^{\frac{1}{q}}.
\end{multline}
Since $\bigcup_{j=1}^{m} Q_\e(x_j)\subset\Omega$ and
$Q_\e(x_k)\cap Q_\e(x_j)$ whenever $k\neq j$, by
\er{fgyufghfghjgghgjkhkkhhkggkjgmhkjjuyhghgGHKKvbvbvvzzjjjkkpkhjhjjhhjhkuuiuibvq}
we finally obtain
\begin{multline}\label{fgyufghfghjgghgjkhkkhhkggkjgmhkjjuyhghgGHKKvbvbvvzzjjjkkpkhjhjjhhjhkuuiuiyuyyubvq}
\sum_{j=1}^{m}\e^{N-1}\Bigg(\frac{1}{\big(\mathcal{L}^N(Q_\e(x_j))\big)^2}\int_{Q_\e(x_j)}\int_{Q_\e(x_j)}\Big|u(y)-u(x)\Big|dydx\Bigg)
\\ \leq
N^{\frac{N+1}{2q}}\Bigg(\int\limits_{\bigcup_{j=1}^{m}Q_\e(x_j)}\bigg(\int\limits_{B_{(\e\sqrt{N})}(x)\cap\Omega}\frac{1}{(\e\sqrt{N})^{N}}\,\frac{\big|u(y)-u(x)\big|^q}{|y-x|}
dy\bigg)dx\Bigg)^{\frac{1}{q}}\\ \leq
N^{\frac{N+1}{2q}}\Bigg(\int\limits_{\Omega}\int\limits_{B_{(\e\sqrt{N})}(x)\cap\Omega}\frac{1}{(\e\sqrt{N})^{N}}\,\frac{\big|u(y)-u(x)\big|^q}{|y-x|}
dydx\Bigg)^{\frac{1}{q}}.
\end{multline}
\end{proof}
From the above we can now deduce the main results about $BV^q$-spaces
as stated in the Introduction.
\begin{proof}[Proof of Proposition \ref{ghgghgghzzbvq}]
Follows from Lemma \ref{hjhhjKKzzbvq}.
\end{proof}
%\begin{proof}[Proof of Theorem \ref{ghgghgghjjkjkzzbvq}] This is a  particular case of Theorem~\ref{ghgghgghjjkjkzzbvqKK}.
%\end{proof}
\begin{proof}[Proof of Theorem \ref{huyhuyuughhjhhjjhhjjjhhjhjjhhhjzzbvq}]
For $q=1$ it is well known. On the other hand, for $q>1$ the results
follow from Lemma~\ref{huyhuyuughhjhhjjhhjjjzzbvq}.
\end{proof}
\begin{proof}[Proof of Theorem \ref{bhugggyffyfbvq}]
Follows from Lemma~\ref{guytuttbvq} and Definition~\ref{gvyfhgfhfffbvq}.
\end{proof}

\section{An application to Aviles-Giga type energies: proof of Theorem
\ref{hgughgfhfzzbvq}}\label{AVGG}
 Th main ingredient needed for the proof of
 Theorem~\ref{hgughgfhfzzbvq} is given by the next Lemma.
\begin{lemma}\label{hgughgfhfzzbvqaa}
Let $\Omega,\Omega_0\subset\R^N$ be two open sets such that
$\Omega_0\subset\subset\Omega$.
Let $q>1$ and  $\psi\in W^{1,\infty}_{loc}(\Omega,\R)$ be such that
$|\nabla\psi(x)|=1$ for a.e. $x\in\Omega$ and $\nabla\psi(x)\in
BV^q_{loc}(\Omega,\R^N)$. For $\eta\in C^\infty_c(\R^N,[0,\infty))$
 satisfying  $\supp\eta\subset \ov B_1(0)$ and
$\int_{\R^N}\eta(z)dz=1$, every $x\in\Omega$ and every
$0<\e<\dist(x,\partial\Omega)$ define
\begin{equation}\label{jmvnvnbccbvhjhjhhjjkhgjgGHKKjhhjzzbvqaa}
\psi_\e(x):=\frac{1}{\e^N}\int_{\R^N}\eta\Big(\frac{y-x}{\e}\Big)\psi(y)dy=\int_{\R^N}\eta(z)\psi(x+\e
z)dz.
\end{equation}
Then,
\begin{equation}\label{fgyufghfghjgghgjkhkkhhkggkjgmhkjjuyhghgGHKKvbvbvvhughghojjlkjjjlouokjjzzbvqaajjkljk}
\limsup_{\e\to
0^+}\int_{\Omega_0}\e^{q-1}\big|\nabla^2\psi_\e(x)\big|^qdx\leq
\bigg(\int_{\R^N}|z|^{\frac{1}{q-1}}\big|\nabla\eta(z)\big|^{\frac{q}{q-1}}dz\bigg)^{q-1}A_{\nabla\psi,q}(\ov\Omega_0).
%\Bigg(\int_{B_1(0)}\int_{\Omega_0}\frac{1}{\e|z|}\Big|\nabla\psi(x+\e z)-\nabla\psi(x)\Big|^qdxdz\Bigg).
\end{equation}
Moreover, if $q\geq 2$ then
\begin{equation}\label{fgyufghfghjgghgjkhkkhhkggkjgmhkjjuyhghgGHKKvbvbvvhughghojjlkjjjlouijljkjklhhjojuihhhjkzzbvqaakkgfhpll}
\limsup_{\e\to
0^+}\int_{\Omega_0}\frac{1}{\e}\Big(1-\big|\nabla\psi_\e(x)\big|^2\Big)^{\frac{q}{2}}dx\leq
\bigg(\int_{\R^N}|z|^{\frac{2}{q-2}}\big|\eta(z)\big|^{\frac{q}{q-2}}dz\bigg)^{\frac{q-2}{2}}A_{\nabla\psi,q}(\ov\Omega_0).
%\Bigg(\int_{B_1(0)}\int_{\Omega_0}\frac{1}{\e|z|}\Big|\nabla\psi(x+\e z)-\nabla\psi(x)\Big|^qdxdz\Bigg).
\end{equation}
\end{lemma}
\begin{proof}
For every $x\in\Omega_0$ and small enough $\e>0$ we have
\begin{equation}\label{jmvnvnbccbvhjhjhhjjkhgjgGHKKjhhjuyugzzbvqaa}
\nabla\psi_\e(x):=\frac{1}{\e^N}\int_{\R^N}\eta\Big(\frac{y-x}{\e}\Big)\nabla\psi(y)dy=\int_{\R^N}\eta(z)\nabla\psi(x+\e
z)dz,
\end{equation}
and
\begin{equation}\label{jmvnvnbccbvhjhjhhjjkhgjgGHKKjhhjuyuguyutuyzzbvqaa}
\e\nabla^2\psi_\e(x):=-\frac{1}{\e^N}\int_{\R^N}\nabla\eta\Big(\frac{y-x}{\e}\Big)\otimes\nabla\psi(y)dy=-\int_{\R^N}\nabla\eta(z)\otimes\nabla\psi(x+\e
z)dz.
\end{equation}
By \er{jmvnvnbccbvhjhjhhjjkhgjgGHKKjhhjuyuguyutuyzzbvqaa},
\begin{multline}\label{fgyufghfghjgghgjkhkkhhkggkjgmhkjjuyhghgGHKKvbvbvvhughghojjlkjjjlouzzbvqaa}
\int_{\Omega_0}\e^{q-1}\big|\nabla^2\psi_\e(x)\big|^qdx=\int_{\Omega_0}\frac{1}{\e}\big|\e\nabla^2\psi_\e(x)\big|^qdx
=\int_{\Omega_0}\frac{1}{\e}\Bigg|\int_{\R^N}\nabla\eta(z)\otimes\nabla\psi(x+\e
z)dz\Bigg|^qdx\\
=\int_{\Omega_0}\frac{1}{\e}\Bigg|\int_{\R^N}\nabla\eta(z)\otimes\Big(\nabla\psi(x+\e
z)-\nabla\psi(x)\Big)dz\Bigg|^qdx.
\end{multline}
From
\er{fgyufghfghjgghgjkhkkhhkggkjgmhkjjuyhghgGHKKvbvbvvhughghojjlkjjjlouzzbvqaa}
and  H\"{o}lder inequality we finally deduce that
\begin{multline}\label{fgyufghfghjgghgjkhkkhhkggkjgmhkjjuyhghgGHKKvbvbvvhughghojjlkjjjlouokjjzzbvqaa}
\int_{\Omega_0}\e^{q-1}\big|\nabla^2\psi_\e(x)\big|^qdx=
\int_{\Omega_0}\frac{1}{\e}\Bigg|\int_{\R^N}|z|^{\frac{1}{q}}\nabla\eta(z)\otimes\frac{1}{|z|^{\frac{1}{q}}}\Big(\nabla\psi(x+\e
z)-\nabla\psi(x)\Big)dz\Bigg|^qdx\leq\\
\bigg(\int_{\R^N}|z|^{\frac{1}{q-1}}\big|\nabla\eta(z)\big|^{\frac{q}{q-1}}dz\bigg)^{q-1}\Bigg(\int_{B_1(0)}
\int_{\Omega_0}\frac{1}{\e|z|}\Big|\nabla\psi(x+\e
z)-\nabla\psi(x)\Big|^qdxdz\Bigg),
\end{multline}
and
\er{fgyufghfghjgghgjkhkkhhkggkjgmhkjjuyhghgGHKKvbvbvvhughghojjlkjjjlouokjjzzbvqaajjkljk}
follows.

On the other hand, since $|\nabla\psi|^2=1$ a.e.~in $\Omega$ we may write
\begin{multline}\label{fgyufghfghjgghgjkhkkhhkggkjgmhkjjuyhghgGHKKvbvbvvhughghojjlkjjjlouijljkjklhhjzzbvqaa}
\int_{\Omega_0}\frac{1}{\e}\Big(1-\big|\nabla\psi_\e(x)\big|^2\Big)^{\frac{q}{2}}dx=\int_{\Omega_0}\frac{1}{\e}
\Bigg(1-\bigg|\int_{\R^N}\eta(z)\nabla\psi(x+\e
z)dz\bigg|^2\Bigg)^{\frac{q}{2}}dx=\\ \int_{\Omega_0}\frac{1}{\e}
\Bigg(\int_{\R^N}\eta(z)\big|\nabla\psi(x+\e
z)\big|^2\,dz-\bigg|\int_{\R^N}\eta(z)\nabla\psi(x+\e
z)dz\bigg|^2\Bigg)^{\frac{q}{2}}dx\,.
\end{multline}
By elementary computations we find for every $x\in\Omega_0$,
\begin{multline}
  \label{eq:2}
\int_{\R^N}\eta(z)\big|\nabla\psi(x+\e
z)\big|^2\,dz-\bigg|\int_{\R^N}\eta(z)\nabla\psi(x+\e
z)dz\bigg|^2=\\ \int_{\R^N}\eta(z)\bigg|\nabla\psi(x+\e z)-\int_{\R^N}\eta(y)\nabla\psi(x+\e y)dy\bigg|^2\,dz=\\
\int_{\R^N}\eta(z)\bigg|\nabla\psi(x+\e
z)-\nabla\psi(x)-\int_{\R^N}\eta(y)\big(\nabla\psi(x+\e
y)-\nabla\psi(x)\big)dy\bigg|^2\,dz=\\
\int_{\R^N}\eta(z)\Big|\nabla\psi(x+\e
z)-\nabla\psi(x)\Big|^2-\bigg(\int_{\R^N}\eta(z)\big(\nabla\psi(x+\e
z)-\nabla\psi(x)\big)dz\bigg)^2\,.
\end{multline}
Plugging \eqref{eq:2} in
\er{fgyufghfghjgghgjkhkkhhkggkjgmhkjjuyhghgGHKKvbvbvvhughghojjlkjjjlouijljkjklhhjzzbvqaa},
and then applying
H\"{o}lder inequality  (using $q\geq
2$) yields
\begin{multline}\label{fgyufghfghjgghgjkhkkhhkggkjgmhkjjuyhghgGHKKvbvbvvhughghojjlkjjjlouijljkjklhhjojuihhhjkzzbvqaa}
\int_{\Omega_0}\frac{1}{\e}\Big(1-\big|\nabla\psi_\e(x)\big|^2\Big)^{\frac{q}{2}}dx\leq
\int_{\Omega_0}\frac{1}{\e}
\Bigg(\int_{\R^N}\eta(z)\Big|\nabla\psi(x+\e
z)-\nabla\psi(x)\Big|^2dz\Bigg)^{\frac{q}{2}}dx=\\
\int_{\Omega_0}\frac{1}{\e}
\Bigg(\int_{\R^N}|z|^{\frac{2}{q}}\eta(z)\frac{1}{|z|^{\frac{2}{q}}}\Big|\nabla\psi(x+\e
z)-\nabla\psi(x)\Big|^2dz\Bigg)^{\frac{q}{2}}dx\leq\\
\bigg(\int_{\R^N}|z|^{\frac{2}{q-2}}\big|\eta(z)\big|^{\frac{q}{q-2}}dz\bigg)^{\frac{q-2}{2}}\Bigg(\int_{B_1(0)}
\int_{\Omega_0}\frac{1}{\e|z|}\Big|\nabla\psi(x+\e
z)-\nabla\psi(x)\Big|^qdxdz\Bigg).
\end{multline}
Passing to the limit $\e\to 0^+$ in
\er{fgyufghfghjgghgjkhkkhhkggkjgmhkjjuyhghgGHKKvbvbvvhughghojjlkjjjlouijljkjklhhjojuihhhjkzzbvqaa}
gives immediately
\er{fgyufghfghjgghgjkhkkhhkggkjgmhkjjuyhghgGHKKvbvbvvhughghojjlkjjjlouijljkjklhhjojuihhhjkzzbvqaakkgfhpll}.
\end{proof}

\begin{proof}[Proof of Theorem \ref{hgughgfhfzzbvq}]
Inequality
\er{fgyufghfghjgghgjkhkkhhkggkjgmhkjjuyhghgGHKKvbvbvvhughghojjlkjjjlouokjjzzbvqaajjkljkjjkjkjkjkh}
follows from Lemma \ref{hgughgfhfzzbvqaa}.
%Let $\Omega_0\subset\Omega\subset\R^N$ and $\psi_\e(x):\Omega\to\R$.
Next by H\"{o}lder inequality we have:
\begin{multline}\label{fgyufghfghjgghgjkhkkhhkggkjgmhkjjuyhghgGHKKvbvbvvzzbvq}
\int_{\Omega_0}\e^2\big|\nabla^2\psi_\e(x)\big|^3dx+\int_{\Omega_0}\frac{1}{\e}\Big|1-\big|\nabla\psi_\e(x)\big|^2\Big|^{\frac{3}{2}}dx=\\
\frac{1}{\e}\int_{\Omega_0}\frac{1}{3}\Big(\sqrt[3]{3}\big|\e\nabla^2\psi_\e(x)\big|\Big)^3dx+\frac{1}{\e}\int_{\Omega_0}\frac{2}{3}\bigg(\sqrt[3]{\frac{9}{4}}
\Big|1-\big|\nabla\psi_\e(x)\big|^2\Big|\bigg)^{\frac{3}{2}}dx\\
\geq\int_{\Omega_0}\frac{1}{\e}\Big(\sqrt[3]{3}\big|\e\nabla^2\psi_\e(x)\big|\Big)\bigg(\sqrt[3]{\frac{9}{4}}
\Big|1-\big|\nabla\psi_\e(x)\big|^2\Big|\bigg)dx=\frac{3}{\sqrt[3]{4}}\int_{\Omega_0}\big|\nabla^2\psi_\e(x)\big|\Big|1-\big|\nabla\psi_\e(x)\big|^2\Big|dx.
\end{multline}
Thus we deduce the first inequality in
\er{fgyufghfghjgghgjkhkkhhkggkjgmhkjjuyhghgGHKKvbvbvvhughghojjlkjjjlouokjjgghgjgokjjlkzzbvq}.
On the other hand, the second inequality in
\er{fgyufghfghjgghgjkhkkhhkggkjgmhkjjuyhghgGHKKvbvbvvhughghojjlkjjjlouokjjgghgjgokjjlkzzbvq}
is just a special case of
\er{fgyufghfghjgghgjkhkkhhkggkjgmhkjjuyhghgGHKKvbvbvvhughghojjlkjjjlouokjjzzbvqaajjkljkjjkjkjkjkh}
for $q=p=3$.
\end{proof}

\appendix

\section{Appendix:
% A
Some known results on BV-spaces}\label{AppA}
In what follows we
present some known definitions and results on BV-spaces; some of them
were used in the previous sections.
We rely mainly on the book
\cite{amb} by Ambrosio, Fusco and Pallara.
\begin{definition}
Let $\Omega$ be a domain in $\R^N$ and let $f\in L^1(\Omega,\R^m)$.
We say that $f\in BV(\Omega,\R^m)$ if the following quantity is
finite:
\begin{equation*}
\int_\Omega|Df|:= \sup\bigg\{\int_\Omega f\cdot\Div\f \,dx :\,\f\in
C^1_c(\Omega,\R^{m\times N}),\;|\f(x)|\leq 1\;\forall x \bigg\}.
\end{equation*}
\end{definition}
\begin{definition}\label{defjac889878}
Let $\Omega$ be a domain in $\R^N$. Consider a function
$f\in L^1_{loc}(\Omega,\R^m)$ and a point $x\in\Omega$.\\
i) We say that $x$ is an {\em approximate continuity point} of $f$
if there exists $z\in\R^m$ such that
\begin{equation*}
%\label{hiygutguttu}
\lim\limits_{\rho\to 0^+}\frac{\int_{B_\rho(x)}|f(y)-z|\,dy}
{\rho^N}=0.
\end{equation*}
In this case we denote $z$ by $\tilde{f}(x)$. The set of approximate
continuity points of
$f$ is denoted by $G_f$.\\
ii) We say that $x$ is an {\em approximate jump point} of $f$ if
there exist $a,b\in\R^m$ and $\vec\nu\in S^{N-1}$ such that $a\neq
b$ and
%\begin{equation*}
%%%\label{aprplmin}
%\lim\limits_{\rho\to
%0^+}\frac{\int_{B_\rho^+(x,\vec\nu)}|f(y)-a|\,dy}
%{\mathcal{L}^N\big(B_\rho(x)\big)}=0,\quad \lim\limits_{\rho\to
%0^+}\frac{\int_{B_\rho^-(x,\vec\nu)}|f(y)-b|\,dy}
%{\mathcal{L}^N\big(B_\rho(x)\big)}=0.
%\end{equation*}
\begin{equation*}
\lim\limits_{\rho\to
0^+}\frac{\int_{B_\rho(x)}\big|\,f(y)-\chi(a,b,\vec\nu)(y)\,\big|\,dy}{\rho^N}=0,
\end{equation*}
where $\chi(a,b,\vec\nu)$ is defined by
\begin{equation*}
%%%\label{GYGYGY}
\chi(a,b,\vec\nu)(y):=
\begin{cases}
b\quad\text{if }\vec\nu\cdot y<0,\\
a\quad\text{if }\vec\nu\cdot y>0.
\end{cases}
\end{equation*}
The triple $(a,b,\vec\nu)$, uniquely determined, up to a permutation
of $(a,b)$ and a change of sign of $\vec\nu$, is denoted by
$(f^+(x),f^-(x),\vec\nu_f(x))$. We shall call $\vec\nu_f(x)$ the
{\em approximate jump vector} and we shall sometimes write simply
$\vec\nu(x)$ if the reference to the function $f$ is clear. The set
of approximate jump points is denoted by $J_f$. A choice of
$\vec\nu(x)$ for every $x\in J_f$ determines an orientation of
$J_f$. At an approximate continuity point $x$, we shall use the
convention $f^+(x)=f^-(x)=\tilde f(x)$.
\end{definition}
\begin{theorem}[Theorems 3.69 and 3.78 from \cite{amb}]\label{petTh}
Consider an open set $\Omega\subset\R^N$ and $f\in BV(\Omega,\R^m)$.
Then:\\
\noindent i) $\mathcal{H}^{N-1}$-a.e. point in
$\Omega\setminus J_f$ is a point of approximate continuity of $f$.\\
\noindent ii) The set $J_f$  is
$\sigma$-$\mathcal{H}^{N-1}$-rectifiable Borel set, oriented by
$\vec\nu(x)$. I.e., the set $J_f$ is $\mathcal{H}^{N-1}$
$\sigma$-finite, there exist countably many $C^1$ hypersurfaces
$\{S_k\}^{\infty}_{k=1}$ such that
$\mathcal{H}^{N-1}\Big(J_f\setminus\bigcup\limits_{k=1}^{\infty}S_k\Big)=0$,
and for $\mathcal{H}^{N-1}$-a.e. $x\in J_f\cap S_k$, the approximate
jump vector $\vec\nu(x)$ is  normal to $S_k$ at the point $x$.
\\ \noindent iii)
$\big[(f^+-f^-)\otimes\vec\nu_f\big](x)\in
L^1(J_f,d\mathcal{H}^{N-1})$.
\end{theorem}
\begin{theorem}[Theorems 3.92 and 3.78 from \cite{amb}]\label{vtTh3}
Consider an open set $\Omega\subset\R^N$ and $f\in BV(\Omega,\R^m)$.
Then the distributional gradient
 $D f$ can be decomposed
as a sum of two Borel regular finite matrix-valued measures $\mu_f$
and $D^j f$ on $\Omega$,
\begin{equation*}
D f=\mu_f+D^j f,
\end{equation*}
where
\begin{equation*}
D^j f=(f^+-f^-)\otimes\vec\nu_f \mathcal{H}^{N-1}\llcorner J_f
\end{equation*}
is called the jump part of $D f$ and
\begin{equation*}
\mu_f=(D^a f+D^c f)
\end{equation*}
is a sum of the absolutely continuous and the Cantor parts of $D f$.
The two parts $\mu_f$ and $D^j f$ are mutually singular to each
other. Moreover, $\mu_f (B)=0$ for any Borel set $B\subset\Omega$
which is $\mathcal{H}^{N-1}$ $\sigma$-finite.
\end{theorem}

%
%
%
\begin{comment}
\begin{theorem}[Theorems 3.92 and 3.78 from \cite{amb}]\label{vtTh3}
The distributional gradient
 $D f$ can be decomposed
as a sum of three Borel regular finite matrix-valued measures on
$\Omega$,
\begin{equation*}
D f=D^a f+D^c f+D^j f
\end{equation*}
with
\begin{equation*}
 D^a f=(\nabla f)\,\mathcal{L}^N ~\text{ and }~ D^j f=(f^+-f^-)\otimes\vec\nu_f
\mathcal{H}^{N-1}\llcorner J_f\,.
\end{equation*}
$D^a$, $D^c$ and $D^j$ are called absolutely continuous part, Cantor
and jump part of $D f$, respectively, and  $\nabla f\in
L^1(\Omega,\R^{m\times N})$ is the approximate differential of $f$.
The three parts are mutually singular to each  other. Moreover,
%we have the following properties:\\
%i) $(D^a f) (A)=0$ for every Borel set
%$A\subset\Omega$ such that $\mathcal{L}^N(A)=0$;
%i)
the support of $D^cf$ is concentrated on a set of
$\mathcal{L}^N$-measure zero, but $(D^c f) (B)=0$ for any Borel set
$B\subset\Omega$ which is $\sigma$-finite with respect to
$\mathcal{H}^{N-1}$.
%\\ii) $[D^a f]\big(f^{-1}(H)\big)=0$ and $[D^c f]\big(\tilde f^{-1}(H)\big)=0$ for every $H\subset\R^m$ satisfying $\mathcal{H}^1(H)=0$.
\end{theorem}
\end{comment}
%
%
%

\section{Appendix:
% A
Proof of Proposition \ref{gygfhffghlkoii}}\label{AppB}
\begin{lemma}\label{gygfhffghlk}
For every $q\geq 1$, if a measurable function $f:\R\to\R^d$
defined a.e. in $\R$ belongs to the space $\hat V_q(\R,\R^d)$,
%has an essential bounded $q$-variation in $\R$,
then $f\in BV^q_{loc}\big(\R,\R^d\big)$. Moreover, we have:
\begin{equation}\label{vhghfggfgfgfdfdiuuiuihjhjlklk}
\bar A_{f,q}\big(\R\big)\leq 4\big(\hat v_{q,\R}(f)\big)^q.
\end{equation}
\end{lemma}
\begin{proof}
First, assume that $f:\R\to\R^d$ is defined {\em everywhere}
in $\R$ and satisfies  $v_{q,\R}(f)<\infty$. Then by
\er{GMT'3jGHKKkkhjjhgzzZZzzZZzzbvq} we have:
\begin{multline}\label{GMT'3jGHKKkkhjjhgzzZZzzZZzzbvqkkklkk}
\bar
A_{f,q}\big(\R\big)=\sup\limits_{\e\in(0,1)}\Bigg(\int\limits_\R\int\limits_{x-\e}^{x+\e}\frac{\big|f(
y)-f(x)\big|^q}{\e|y-x|}dydx\Bigg)=
\sup\limits_{\e\in(0,1)}\Bigg(\int\limits_{-1}^{1}\int\limits_{\R}\frac{\big|f(x+\e
z)-f(x)\big|^q}{\e|z|}dxdz\Bigg)=\\
\sup\limits_{\e\in(0,1)}\Bigg(\int\limits_{-1}^{1}\sum\limits_{n=0}^{\infty}\bigg(\int\limits_{n\e|z|}^{(n+1)\e|z|}\frac{\big|f(x+\e
z)-f(x)\big|^q}{\e|z|}dx+\!\!\int\limits_{-(n+1)\e|z|}^{-n\e|z|}\frac{\big|f(x+\e
z)-f(x)\big|^q}{\e|z|}dx\bigg)dz\Bigg)\leq\\
\sup\limits_{\e\in(0,1)}\Bigg\{
\lim\limits_{n\to\infty}\int\limits_{-1}^{1}\sum\limits_{k=0}^{n}\frac{1}{\e|z|}\bigg(\int\limits_{2k\e|z|}^{(2k+1)\e|z|}\big|f(x+\e
z)-f(x)\big|^qdx+\!\!\!\int\limits_{-(2k+2)\e|z|}^{-(2k+1)\e|z|}\big|f(x+\e
z)-f(x)\big|^qdx\bigg)dz\Bigg\}+\\
\sup\limits_{\e\in(0,1)}\Bigg\{\lim\limits_{n\to\infty}\int\limits_{-1}^{1}\sum\limits_{k=0}^{n}\frac{1}{\e|z|}\bigg(\int\limits_{(2k+1)\e|z|}^{(2k+2)\e|z|}\big|f(x+\e
z)-f(x)\big|^q dx+\!\!\int\limits_{-(2k+1)\e|z|}^{-2k\e|z|}\big|f(x+\e
z)-f(x)\big|^qdx\bigg)dz\Bigg\}\,.
\end{multline}
Denoting $J_m=(m\e|z|,(m+1)\e|z|)$, we get from
\er{GMT'3jGHKKkkhjjhgzzZZzzZZzzbvqkkklkk} that
% %
% %
% %
\begin{equation}
\label{GMT'3jGHKKkkhjjhgzzZZzzZZzzbvqkkklkkhuuyu}
 \begin{aligned}
 A_{f,q}\big(\R\big)&\leq
 \sup\limits_{\e\in(0,1)}\Bigg\{\lim\limits_{n\to\infty}\int\limits_{-1}^{1}\sum\limits_{k=0}^{n}\Bigg(\sup\limits_{x\in
   J_{2k}}\big|f(x+\e z)-f(x)\big|^q+\sup\limits_{x\in  J_{-(2k+2)}}\!\big|f(x+\e
 z)-f(x)\big|^q\Bigg)dz\Bigg\}\\ &+
 \sup\limits_{\e\in(0,1)}\Bigg\{\lim\limits_{n\to\infty}\int\limits_{-1}^{1}\sum\limits_{k=0}^{n}\Bigg(\sup\limits_{x\in  J_{2k+1}}\big|f(x+\e
 z)-f(x)\big|^q+ \sup\limits_{x\in J_{-(2k+1)}}\big|f(x+\e
 z)-f(x)\big|^q\Bigg)dz\Bigg\} \\
 &\leq 4\big( v_{q,\R}(f)\big)^q.
 \end{aligned}
\end{equation}

In the general case we have, by
\er{GMT'3jGHKKkkhjjhgzzZZzzZZzzbvqkkklkkhuuyu}, for every
$g:\R\to\R^d$ (defined everywhere on $\R$) satisfying
$g(x)=f(x)$ a.e. in $\R$
\begin{equation}\label{GMT'3jGHKKkkhjjhgzzZZzzZZzzbvqkkklkkyygjhihhi}
\bar A_{f,q}\big(\R\big)=\bar A_{g,q}\big(\R\big)\leq 4\big(
v_{q,\R}(g)\big)^q.
\end{equation}
Thus, taking infimum of the r.h.s.~of
\er{GMT'3jGHKKkkhjjhgzzZZzzZZzzbvqkkklkkyygjhihhi} over all such $g$'s
we finally deduce \er{vhghfggfgfgfdfdiuuiuihjhjlklk}.
\end{proof}
\begin{proof}[Proof of Proposition \ref{gygfhffghlkoii}]
Let $g:[a,b]\to\R^d$ be defined {\em everywhere} in
$[a,b]$, satisfying $g(x)=f(x)$ a.e..~in $[a,b]$ and
$v_{q,[a,b]}(g)<\infty$. Consider $\tilde g:\R\to\R^d$ defined
by
\begin{equation}\label{GMT'3jGHKKkkhjjhgzzZZzzZZzzbvqkkklkkyygjhihhijkjk}
\tilde g(x):=\begin{cases} g(x)\quad\quad \forall\,x\in [a,b],\\
g(a)\quad\quad \forall\,x\in (-\infty,a),\\
g(b)\quad\quad \forall\,x\in (b,\infty).
\end{cases}
\end{equation}
By the definition of $v_{q,I}$ in \er{vhghfggfgfgfdfdiuuiui} we
clearly have
\begin{equation}\label{GMT'3jGHKKkkhjjhgzzZZzzZZzzbvqkkklkkyygjhihhiljjljjkjk}
v_{q,\R}(\tilde g)=v_{q,[a,b]}(g)\,.
\end{equation}
Combining \er{GMT'3jGHKKkkhjjhgzzZZzzZZzzbvqkkklkkhuuyu} with
\er{GMT'3jGHKKkkhjjhgzzZZzzZZzzbvqkkklkkyygjhihhiljjljjkjk} we
obtain
\begin{equation}\label{GMT'3jGHKKkkhjjhgzzZZzzZZzzbvqkkklkkyygjhihhiljjljjkjkklkkjl}
\bar A_{f,q}\big((a,b)\big)=\bar A_{g,q}\big((a,b)\big)=\bar
A_{\tilde g,q}\big((a,b)\big)\leq\bar A_{\tilde g,q}\big(\R\big)\leq
4\big(v_{q,\R}(\tilde g)\big)^q=4\big(v_{q,[a,b]}(g)\big)^q.
\end{equation}
Taking the infimum of the r.h.s. of
\er{GMT'3jGHKKkkhjjhgzzZZzzZZzzbvqkkklkkyygjhihhiljjljjkjkklkkjl}
over all $g$'s as above we finally deduce
\er{vhghfggfgfgfdfdiuuiuihjhjlklkouiui} and that $f\in
BV^q\big((a,b),\R^d\big)$.
\end{proof}

\begin{proof}[Proof of Lemma \ref{hjgjg}]
We have,
\begin{multline}\label{gjhgjhfghffhnmmbbbkhjhh}
\sup\limits_{\rho\in(0,\infty)}\Bigg(\sup_{|h|=\rho}\int_{\mathbb{R}^N}\bigg(\frac{1}{\rho^s}\big|u(x+h)-u(x)\big|\bigg)^qdx\Bigg)\leq
\sup\limits_{\rho\in(0,\infty)}\Bigg(\sup_{|h|\leq\rho}\int_{\mathbb{R}^N}\bigg(\frac{1}{\rho^s}\big|u(x+h)-u(x)\big|\bigg)^qdx\Bigg)\\=
\sup\limits_{\rho\in(0,\infty)}\sup_{t\in(0,\rho]}\Bigg(\sup_{|h|=t}\int_{\mathbb{R}^N}\bigg(\frac{1}{\rho^s}\big|u(x+h)-u(x)\big|\bigg)^qdx\Bigg)\\
\leq
\sup\limits_{\rho\in(0,\infty)}\Bigg(\sup_{t\in(0,\rho]}\bigg(\sup_{|h|=t}\int_{\mathbb{R}^N}\Big(\frac{1}{t^s}\big|u(x+h)-u(x)\big|\Big)^qdx\bigg)\Bigg)
=\\
\sup\limits_{\rho\in(0,\infty)}\Bigg(\sup_{|h|=\rho}\int_{\mathbb{R}^N}\bigg(\frac{1}{\rho^s}\big|u(x+h)-u(x)\big|\bigg)^qdx\Bigg)
=\\
\sup\limits_{\rho\in(0,\infty)}\Bigg(\sup_{\vec k\in
S^{N-1}}\int_{\mathbb{R}^N}\bigg(\frac{1}{\rho^s}\big|u(x+\rho\vec
k)-u(x)\big|\bigg)^qdx\Bigg).
\end{multline}
In particular, for $s=\frac{1}{q}$, by \er{gjhgjhfghffhnmmbbbkhjhh}
we deduce
\begin{multline}\label{gjhgjhfghffhnmmbbbkhjhhkljkjk}
\sup\limits_{\rho\in(0,\infty)}\Bigg(\sup_{|h|\leq\rho}\int_{\mathbb{R}^N}\bigg(\frac{1}{\rho^s}\big|u(x+h)-u(x)\big|\bigg)^qdx\Bigg)=\\
\sup\limits_{\rho\in(0,\infty)}\bigg(\sup_{\vec k\in
S^{N-1}}\int_{\mathbb{R}^N}\frac{1}{\rho}\big|u(x+\rho\vec
k)-u(x)\big|^qdx\bigg).
\end{multline}
On the other hand by the triangle inequality and the convexity of
$g(s):=|s|^q$ for every $\delta>0$ we have,
\begin{multline}\label{gjhgjhfghffhj}
\sup\limits_{\rho\in(0,\delta)}\bigg(\sup_{\vec k\in
S^{N-1}}\int_{\mathbb{R}^N}\frac{1}{\rho}\big|u(x+\rho\vec
k)-u(x)\big|^qdx\bigg)\leq\sup\limits_{\rho\in(0,\infty)}\bigg(\sup_{\vec
k\in S^{N-1}}\int_{\mathbb{R}^N}\frac{1}{\rho}\big|u(x+\rho\vec
k)-u(x)\big|^qdx\bigg)\leq\\
\sup\limits_{\rho\in(0,\delta)}\bigg(\sup_{\vec k\in
S^{N-1}}\int_{\mathbb{R}^N}\frac{1}{\rho}\big|u(x+\rho\vec
k)-u(x)\big|^qdx\bigg)+\sup\limits_{\rho\in[\delta,\infty)}\bigg(\sup_{\vec
k\in S^{N-1}}\int_{\mathbb{R}^N}\frac{1}{\rho}\big|u(x+\rho\vec
k)-u(x)\big|^qdx\bigg)\\ \leq
\sup\limits_{\rho\in(0,\delta)}\bigg(\sup_{\vec k\in
S^{N-1}}\int_{\mathbb{R}^N}\frac{1}{\rho}\big|u(x+\rho\vec
k)-u(x)\big|^qdx\bigg)\\+\frac{2^{q-1}}{\delta}\sup\limits_{\rho\in[\delta,\infty)}\bigg(\sup_{\vec
k\in S^{N-1}}\int_{\mathbb{R}^N}\Big(\big|u(x+\rho\vec
k)\big|^q+\big|u(x)\big|^q\Big)dx\bigg)\\=
\sup\limits_{\rho\in(0,\delta)}\bigg(\sup_{\vec k\in
S^{N-1}}\int_{\mathbb{R}^N}\frac{1}{\rho}\big|u(x+\rho\vec
k)-u(x)\big|^qdx\bigg)+\frac{2^{q}}{\delta}\big\|u\big\|^q_{L^q(\R^N,\R^d)}.
\end{multline}
Therefore, by \er{gjhgjhfghffhnmmbbbkhjhhkljkjk} and
\er{gjhgjhfghffhj} we have:
\begin{multline}\label{gjhgjhfghffhjj}
\sup\limits_{\e\in(0,\delta)}\bigg(\sup_{\vec k\in
S^{N-1}}\int_{\mathbb{R}^N}\frac{1}{\e}\big|u(x+\e\vec
k)-u(x)\big|^qdx\bigg)\leq
\sup\limits_{\rho\in(0,\infty)}\Bigg(\sup_{|h|\leq\rho}\int_{\mathbb{R}^N}\bigg(\frac{1}{\rho^{(1/q)}}\big|u(x+h)-u(x)\big|\bigg)^qdx\Bigg)
\\
\leq \sup\limits_{\e\in(0,\delta)}\bigg(\sup_{\vec k\in
S^{N-1}}\int_{\mathbb{R}^N}\frac{1}{\e}\big|u(x+\e\vec
k)-u(x)\big|^qdx\bigg)+\frac{2^{q}}{\delta}\big\|u\big\|^q_{L^q(\R^N,\R^d)}.
\end{multline}
Thus by \er{gjhgjhfghffhjj} we clearly obtain that if $u\in
L^q(\mathbb{R}^N,\R^d)$ then
\begin{multline}\label{gjhgjhfghffhhgjgjjk}
\sup\limits_{\rho\in(0,\infty)}\Bigg(\sup_{|h|\leq\rho}\int_{\mathbb{R}^N}\bigg(\frac{1}{\rho^{(1/q)}}\big|u(x+h)-u(x)\big|\bigg)^qdx\Bigg)
<\infty\quad\text{if and only if}\\ \quad \limsup\limits_{\e\to
0^+}\bigg(\sup_{\vec k\in
S^{N-1}}\int_{\mathbb{R}^N}\frac{1}{\e}\big|u(x+\e\vec
k)-u(x)\big|^qdx\bigg)<\infty.
\end{multline}
So we proved that $u\in L^q(\R^N,\R^d)$ belongs to
$B_{q,\infty}^{1/q}(\R^N,\R^d)$ if and only if we have $\hat
B_{u,q}\big(\R^N\big)<\infty$.

Next, given open $\Omega\subset\R^N$ let $u\in
L^q_{loc}(\Omega,\R^d)$ and $K\subset\subset\Omega$ be a compact
set. Moreover, consider an open set $U\subset\R^N$ such that we have
the following compact embedding: $$K\subset\subset U\subset\ov
U\subset\subset\Omega.$$ Then, assuming
$u\in\big(B_{q,\infty}^{1/q}\big)_{loc}(\Omega,\R^d)$ implies
existence of $\hat u\in B_{q,\infty}^{1/q}(\mathbb{R}^N,\R^d)$
such that $\hat u(x)= u(x)$ for every $x\in \bar U$, that gives
$$B_{u,q}\big(K\big)=B_{\hat u,q}\big(K\big)\leq \hat B_{\hat u,q}\big(\R^N\big)<\infty.$$
On the other hand, if we assume
\begin{equation}\label{jjhhkghjgjfhfhfh}
B_{u,q}\big(\ov U\big)<+\infty\,,
\end{equation}
then define
\begin{equation}\label{jjhhkghjgjfhfhfhn,nn}
\hat u(x)=\begin{cases}\eta(x)u(x)\quad\forall x\in U\\ 0
\quad\forall x\in \R^N\setminus U,
\end{cases}
\end{equation}
where $\eta(x)\in C^\infty_c\big(U,[0,1]\big)$ is some cut-off
function such that $\eta(x)=1$ for every $x\in K$. Thus in
particular $\hat u(x)= u(x)$ for every $x\in K$ and so, in order to
complete the proof, we need just to show that $\hat u\in
B_{q,\infty}^{1/q}(\mathbb{R}^N,\R^d)$. Thus by
\er{gjhgjhfghffhhgjgjjk} it is sufficient to show:
\begin{equation}\label{gjhgjhfghffhhgjgjjklkhjhhj}
\limsup\limits_{\e\to 0^+}\bigg(\sup_{\vec k\in
S^{N-1}}\int_{\mathbb{R}^N}\frac{1}{\e}\big|\hat u(x+\e\vec k)-\hat
u(x)\big|^qdx\bigg)<\infty.
\end{equation}
However,
%for some open $\hat U\subset\R^N$ we have $\eta(x)=0$ for
%$x\notin\hat U$ together with the following compact embedding: $$\ov
%U\subset\subset\hat U\subset\subset\Omega.$$ Thus
since $|\eta|\leq 1$, $\supp\eta\subset\subset U$ and $\eta$ is
smooth we have:
\begin{multline}\label{gjhgjhfghffhhgjgjjklkhjhhjl;klkl;}
\limsup\limits_{\e\to 0^+}\bigg(\sup_{\vec k\in
S^{N-1}}\int_{\mathbb{R}^N}\frac{1}{\e}\big|\hat u(x+\e\vec k)-\hat
u(x)\big|^qdx\bigg)=\\
\limsup\limits_{\e\to 0^+}\bigg(\sup_{\vec k\in
S^{N-1}}\int_{U}\frac{1}{\e}\big|\eta(x+\e\vec k)u(x+\e\vec
k)-\eta(x) u(x)\big|^qdx\bigg)
=\\
\limsup\limits_{\e\to 0^+}\bigg(\sup_{\vec k\in S^{N-1}}\int_{
U}\frac{1}{\e}\Big|\eta(x+\e\vec k)\big(u(x+\e\vec
k)-u(x)\big)+\big(\eta(x+\e\vec k)-\eta(x)\big) u(x)\Big|^qdx\bigg)
\\
\leq \limsup\limits_{\e\to 0^+}\Bigg(\sup_{\vec k\in S^{N-1}}\int_{
U}\frac{2^{q-1}}{\e}\bigg(\Big|\eta(x+\e\vec k)\big(u(x+\e\vec
k)-u(x)\big)\Big|^q+\Big|\big(\eta(x+\e\vec
k)-\eta(x)\big) u(x)\Big|^q\bigg)dx\Bigg)\\
\leq 2^{q-1}\limsup\limits_{\e\to 0^+}\Bigg(\sup_{\vec k\in
S^{N-1}}\int_{U}\frac{1}{\e}\bigg(\Big|u(x+\e\vec
k)-u(x)\Big|^q\bigg)dx\Bigg)+\\
 2^{q-1}\limsup\limits_{\e\to 0^+}\Bigg(\e^{q-1}\sup_{\vec k\in
S^{N-1}}\int_{U}\bigg|\frac{\big(\eta(x+\e\vec
k)-\eta(x)\big)}{\e}\bigg|^q\big| u(x)\big|^qdx\Bigg)\\
\leq 2^{q-1}B_{u,q}\big(\ov U\big)+2^{q-1}\bigg(\int_{U}\big|
u(x)\big|^qdx\bigg)\big\|\nabla\eta\big\|^q_{L^\infty}<\infty.
\end{multline}
\end{proof}

\end{document}